\documentclass[smallextended,envcountsect,runningheads]{svjour3}

\smartqed

\usepackage{mathptmx}

\journalname{Mathematische Zeitschrift}

\usepackage[leqno]{amsmath}
\usepackage{latexsym,amssymb}
\usepackage[all]{xy} \SelectTips{eu}{} \SilentMatrices
\usepackage{hyperref}

\begin{document}

\title{Algebras that satisfy {A}uslander's condition\\ on vanishing of
  cohomology\thanks{Part of this work was done while L.W.C.\ visited
    the University of Nebraska-Lincoln, partly supported by grants
    from the Danish Natural Science Research Council and the Carlsberg
    Foundation.\\ H.H.\ was partly supported by the Danish Natural
    Science Research Council.}}

%\titlerunning{Auslander's condition on vanishing of
%  cohomology} % if too long for running head

\dedication{In memory of Anders J.~Frankild}

\author{Lars Winther Christensen \and Henrik Holm}

\authorrunning{L.\,W.\ Christensen \and H.\ Holm}

\institute{L.\,W.\ Christensen \at Department of Mathematics and
  Statistics,
  Texas Tech University, Lubbock, TX 79409-1042, U.S.A.\\
  \email{lars.w.christensen@ttu.edu} \and H.\ Holm \at Department of
  Mathematical Sciences, University of Aarhus, DK-8000 {\AA}rhus C,
  Denmark. \\ \emph{Present address:} Department of Basic Sciences and
  Environment, University of Copenhagen,\newline Thorvaldsensvej 40,
  DK-1871 Frederiksberg C, Denmark\\ \email{hholm@life.ku.dk}}

\date{21 January 2009}

\maketitle

\begin{abstract}
  Auslander conjectured that every Artin algebra satisfies a certain
  condition on vanishing of cohomology of finitely generated
  modules. The failure of this conjecture---by a 2003 counterexample
  due to Jorgensen and \c{S}ega---motivates the consideration of the
  class of rings that \textsl{do} satisfy Auslander's condition. We
  call them AC rings and show that an AC Artin algebra that is
  left-Gorenstein is also right-Gorenstein.  Furthermore, the
  Auslander-Reiten Conjecture is proved for AC rings, and Auslander's
  G-dimension is shown to be functorial for AC rings that are
  commutative or have a dualizing complex.


\end{abstract}

%%%%%%%%%%%%%%%%%%%%%%%%%%%% THEOREMSTYLES %%%%%%%%%%%%%%%%%%%%%%%%%%%%

\spnewtheorem*{spthma}{Theorem A}{\bf}{\sl}
\spnewtheorem*{spthmb}{Theorem B}{\bf}{\sl}
\spnewtheorem*{spthmc}{Theorem C}{\bf}{\sl}
\spnewtheorem{spipg}[theorem]{\!\!}{\bf}{\rm}
\spnewtheorem{spthm}[theorem]{Theorem}{\bf}{\sl}
\spnewtheorem{splem}[theorem]{Lemma}{\bf}{\sl}
\spnewtheorem{spcor}[theorem]{Corollary}{\bf}{\sl}
\spnewtheorem{spprp}[theorem]{Proposition}{\bf}{\sl}
\spnewtheorem{sprmk}[theorem]{Remark}{\bf}{\rm}
\spnewtheorem{spobs}[theorem]{Observation}{\bf}{\rm}
\spnewtheorem{spexa}[theorem]{Example}{\bf}{\rm}
\spnewtheorem*{pfofa}{Proof of 4.4}{\it}{\rm}
\spnewtheorem*{pfofac}{Proof of 4.4 continued}{\it}{\rm}
\spnewtheorem*{pac}{Auslander's conjecture}{\bf}{\rm}
\spnewtheorem*{parc}{The Auslander-Reiten Conjecture}{\bf}{\rm}
\spnewtheorem*{ptc}{The Tachikawa Conjectures}{\bf}{\rm}

\newlength{\thmtopspace} %Space above theorem
\newlength{\thmbotspace} %Space below theorem
\newlength{\thmheadspace} %Space after theorem label
\newlength{\thmindent} %For indenting

\setlength{\thmtopspace}%
{0.7\baselineskip plus 0.35\baselineskip minus 0.2\baselineskip}
\setlength{\thmbotspace}%
{0.45\baselineskip plus 0.15\baselineskip minus 0.1\baselineskip}
\setlength{\thmheadspace}{0.5em} \setlength{\thmindent}{0pt}

% Length Controls

\newlength{\thmlistleft} %leftmargin
\newlength{\thmlistright} %rightmargin
\newlength{\thmlistpartopsep} %partopsep
\newlength{\thmlisttopsep} %topsep
\newlength{\thmlistparsep} %parsep
\newlength{\thmlistitemsep} %itemsep

\setlength{\thmlistleft}{2.25em} \setlength{\thmlistright}{0pt}
\setlength{\thmlistitemsep}{0.5ex} \setlength{\thmlistparsep}{0pt}
\setlength{\thmlisttopsep}{1.5\thmlistitemsep}
\setlength{\thmlistpartopsep}{0pt}

%% 1. ENVIRONMENT FOR LISTING EQUIVALENT CONDITIONS

\newcounter{eqc} \newenvironment{eqc}{\begin{list}{\upshape
      (\textit{\roman{eqc}})}%
    {\usecounter{eqc}%
      \setlength{\leftmargin}{\thmlistleft}%
      \setlength{\labelwidth}{\thmlistleft}%
      \setlength{\rightmargin}{\thmlistright}%
      \setlength{\partopsep}{\thmlistpartopsep}%
      \setlength{\topsep}{\thmlisttopsep}%
      \setlength{\parsep}{\thmlistparsep}%
      \setlength{\itemsep}{\thmlistitemsep}}}%
  {\end{list}}%

% Label
\newcommand{\eqclbl}[1]{{\upshape(\textit{#1})}}

% Previous label primed by #1
\newcommand{\preveqc}[1]{{\upshape(\textit{\roman{eqc}#1})}}

% Label with correct space for use in manually generated list
\newcommand{\eqcman}[1]{\eqclbl{#1}\hspace{\labelsep}}

%% 2. ENVIRONMENT FOR LISTING PROPERTIES

\newcounter{prt} \newenvironment{prt}{\begin{list}{\upshape
      (\alph{prt})}%
    {\usecounter{prt}%
      \setlength{\leftmargin}{\thmlistleft}%
      \setlength{\labelwidth}{\thmlistleft}%
      \setlength{\rightmargin}{\thmlistright}%
      \setlength{\partopsep}{\thmlistpartopsep}%
      \setlength{\topsep}{\thmlisttopsep}%
      \setlength{\parsep}{\thmlistparsep}%
      \setlength{\itemsep}{\thmlistitemsep}}}%
  {\end{list}}%

% Label
\newcommand{\prtlbl}[1]{{\upshape(#1)}}

% Previous label primed by #1
\newcommand{\prevprt}[1]{{\upshape(\alph{prt}#1)}}

%% 3. ENVIRONMENT FOR LISTING REQUIREMENTS

\newcounter{rqm} \newenvironment{rqm}{\begin{list}{\upshape
      (\arabic{rqm})}%
    {\usecounter{rqm}%
      \setlength{\leftmargin}{\thmlistleft}%
      \setlength{\labelwidth}{\thmlistleft}%
      \setlength{\rightmargin}{\thmlistright}%
      \setlength{\partopsep}{\thmlistpartopsep}%
      \setlength{\topsep}{\thmlisttopsep}%
      \setlength{\parsep}{\thmlistparsep}%
      \setlength{\itemsep}{\thmlistitemsep}}}%
  {\end{list}}%

%% 5. ENVIRONMENT FOR LISTING ITEMS

\newenvironment{itemlist}{\nopagebreak \begin{list}{$\bullet$}%
    {\setlength{\leftmargin}{\thmlistleft}%
      \setlength{\labelwidth}{\thmlistleft}%
      \setlength{\rightmargin}{\thmlistright}%
      \setlength{\partopsep}{\thmlistpartopsep}%
      \setlength{\topsep}{\thmlisttopsep}%
      \setlength{\parsep}{\thmlistparsep}%
      \setlength{\itemsep}{\thmlistitemsep}}}%
  {\end{list}}%

%%%%%%%%%%%%%%%%%%%%%%% REFERENCES AND
%%%%%%%%%%%%%%%%%%%%%%% CITATIONS %%%%%%%%%%%%%%%%%%%%%%%%

%% 1. REFERENCES TO PARAGRAPHS WITH LABEL INCLUDED

\newcommand{\pgref}[1]{\ref{#1}}
\newcommand{\resref}[2][]{#1\pgref{res:#2}}
\newcommand{\thmref}[2][Theorem~]{#1\pgref{thm:#2}}
\newcommand{\corref}[2][Corollary~]{#1\pgref{cor:#2}}
\newcommand{\prpref}[2][Proposition~]{#1\pgref{prp:#2}}
\newcommand{\lemref}[2][Lemma~]{#1\pgref{lem:#2}}
\newcommand{\obsref}[2][Observation~]{#1\pgref{obs:#2}}
\newcommand{\conref}[2][Construction~]{#1\pgref{con:#2}}
\newcommand{\dfnref}[2][Definition~]{#1\pgref{dfn:#2}}
\newcommand{\exaref}[2][Example~]{#1\pgref{exa:#2}}
\newcommand{\rmkref}[2][Remark~]{#1\pgref{rmk:#2}}
\newcommand{\charef}[2][Chapter~]{#1\ref{cha:#2}}
\newcommand{\secref}[2][Section~]{#1\ref{sec:#2}}

\newcommand{\partpgref}[2]{\ref{#1}\prtlbl{#2}}
\newcommand{\partresref}[2]{\partpgref{res:#1}{#2}}
\newcommand{\partthmref}[3][Theorem~]{#1\partpgref{thm:#2}{#3}}
\newcommand{\partcorref}[3][Corollary~]{#1\partpgref{cor:#2}{#3}}
\newcommand{\partprpref}[3][Proposition~]{#1\partpgref{prp:#2}{#3}}
\newcommand{\partlemref}[3][Lemma~]{#1\partpgref{lem:#2}{#3}}
\newcommand{\partobsref}[3][Observation~]{#1\partpgref{obs:#2}{#3}}
\newcommand{\partdfnref}[3][Definition~]{#1\partpgref{dfn:#2}{#3}}
\newcommand{\partexaref}[3][Example~]{#1\partpgref{exa:#2}{#3}}
\newcommand{\partrmkref}[3][Remark~]{#1\partpgref{rmk:#2}{#3}}

%% 2. REFERENCES TO EXERCISES AND EQUATIONS

\renewcommand{\eqref}[1]{(\ref{eq:#1})}

%% 3. CITATIONS

\newcommand{\rescite}[2][?]{\cite[#1]{#2}}
\newcommand{\thmcite}[2][?]{\cite[thm.~#1]{#2}}
\newcommand{\corcite}[2][?]{\cite[cor.~#1]{#2}}
\newcommand{\prpcite}[2][?]{\cite[prop.~#1]{#2}}
\newcommand{\lemcite}[2][?]{\cite[lem.~#1]{#2}}
\newcommand{\chpcite}[2][?]{\cite[chap.~#1]{#2}}
\newcommand{\seccite}[2][?]{\cite[sec.~#1]{#2}}

%%%%%%%%%%%%%%%%%%%%%%%%%%% AUTHOR COMMENTS %%%%%%%%%%%%%%%%%%%%%%%%%%

\newcommand{\nt}[2][$^\diamondsuit$]{%
  \hspace{0pt}#1\marginpar{\tt\raggedleft #1 #2}}

%%%%%%%%%%%%%%%%%%%%%%%%%%%%%%%% SETS %%%%%%%%%%%%%%%%%%%%%%%%%%%%%%%%%

\newcommand{\set}[2][\,]{\{#1 #2 #1\}}
\newcommand{\setof}[3][\;]{\{#1#2 \mid #3#1\}}
\newcommand{\ZZ}{\mathbb{Z}} \newcommand{\dinZ}{{\d\in\ZZ}}
\newcommand{\hinZ}{{h\in\ZZ}}
\newcommand{\qtext}[1]{\quad\text{#1}\quad}
\newcommand{\qqtext}[1]{\qquad\text{#1}\qquad}
\newcommand{\qand}{\qtext{and}} \newcommand{\qqand}{\qqtext{and}}
\newcommand{\deq}{\:=\:} \newcommand{\dge}{\:\ge\:}
\newcommand{\dis}{\:\is\:} \renewcommand{\a}{\alpha}
\renewcommand{\d}{v} % generic homological degree
\newcommand{\m}{\mathfrak{m}} \newcommand{\p}{\mathfrak{p}}
\newcommand{\is}{\cong} \newcommand{\qis}{\simeq}
\renewcommand{\le}{\leqslant} \renewcommand{\ge}{\geqslant}
\newcommand{\onto}{\twoheadrightarrow}
\newcommand{\lra}{\longrightarrow}
\newcommand{\xla}[2][]{\xleftarrow[#1]{\;#2\;}}
\newcommand{\xra}[2][]{\xrightarrow[#1]{\;#2\;}}
\newcommand{\qla}{\xla{\;\qis\;}} \newcommand{\qra}{\xra{\;\qis\;}}
\newcommand{\dra}[2]{\xra{\dif[#1]{#2}}}
\newcommand{\poly}[2][k]{#1[#2]}
\newcommand{\pows}[2][k]{#1[\mspace{-2.3mu}[#2]\mspace{-2.3mu}]}
\newcommand{\QQ}{\mathbb{Q}} \newcommand{\Rm}{(R,\m)}
\newcommand{\Rhat}{\widehat{R}}
\newcommand{\mapdef}[4][\rightarrow]{\nobreak{#2\colon #3 #1 #4}}
\newcommand{\qisdef}[4][\xra{\qis}]{\nobreak{#2\colon #3 #1 #4}}
\newcommand{\dmapdef}[4][\lra]{\nobreak{#2\colon #3\:#1\:#4}}
\newcommand{\disdef}[4][\ira]{\nobreak{#2\colon #3 #1 #4}}
\newcommand{\dqisdef}[4][\qra]{\nobreak{#2\colon #3 #1 #4}}
\newcommand{\Ker}[1]{\nobreak{\operatorname{Ker}#1}}
\newcommand{\Coker}[1]{\nobreak{\operatorname{Coker}#1}}
\newcommand{\Cone}[1]{\nobreak{\operatorname{Cone}#1}}
\newcommand{\Conep}[1]{\nobreak{(\operatorname{Cone}#1)}}
\newcommand{\tev}[1]{\omega_{#1}} \newcommand{\hev}[1]{\theta_{#1}}
\newcommand{\dif}[2][]{{\partial}^{#2}_{#1}}
\newcommand{\Cy}[2][]{\operatorname{Z}_{#1}(#2)}
\newcommand{\Co}[2][]{\operatorname{C}_{#1}(#2)}
\renewcommand{\H}[2][]{\operatorname{H}_{#1}(#2)}
\newcommand{\Shift}[2][]{\mathsf{\Sigma}^{#1}{#2}}
\newcommand{\Shiftp}[2][]{(\Shift[#1]{#2})}
\newcommand{\Tha}[2]{#2_{{\scriptscriptstyle\le}#1}}
\newcommand{\Thb}[2]{#2_{{\scriptscriptstyle\ge}#1}}
\newcommand{\Tsa}[2]{#2_{{\scriptscriptstyle\subset}#1}}
\newcommand{\Tsb}[2]{#2_{{\scriptscriptstyle\supset}#1}}
\newcommand{\Tsap}[2]{(\Tsa{#1}{#2})}
\newcommand{\lgtR}{\operatorname{length}R}
\newcommand{\codim}[1]{\operatorname{codim}#1}
\newcommand{\SocR}{\operatorname{Soc}R}
\newcommand{\Spec}[1]{\operatorname{Spec}#1}
\newcommand{\lgt}[2][R]{\operatorname{length}_{#1}#2}
\newcommand{\id}[2][R]{\operatorname{id}_{#1}#2}
\newcommand{\pd}[2][R]{\operatorname{pd}_{#1}#2}
\newcommand{\Gdim}[2][R]{\operatorname{G-dim}_{#1}#2}
\newcommand{\Hom}[3][R]{\operatorname{Hom}_{#1}(#2,#3)}
\newcommand{\RHom}[3][R]{\operatorname{\mathbf{R}Hom}_{#1}(#2,#3)}
\newcommand{\Ext}[4][R]{\operatorname{Ext}_{#1}^{#2}(#3,#4)}
\newcommand{\tp}[3][R]{\nobreak{#2\otimes_{#1}#3}}
\newcommand{\tpp}[3][R]{(\tp[#1]{#2}{#3})}
\newcommand{\Ltp}[3][R]{\nobreak{#2\otimes_{#1}^{\mathbf{L}}#3}}
\newcommand{\Ltpp}[3][R]{(\Ltp[#1]{#2}{#3})}
\newcommand{\Tor}[4][R]{\operatorname{Tor}^{#1}_{#2}(#3,#4)}
\newcommand{\Cat}[2]{{\mathsf{#2}}(#1)}
\newcommand{\D}[1][R]{\Cat{#1}{D}}

%%%%%%%%%%%%%%%%%%%%%%%%%%%%% HYPHENATION %%%%%%%%%%%%%%%%%%%%%%%%%%%%%

\hyphenation{mo-dule com-plex com-plex-es mor-phism ho-mo-mor-phism
  iso-mor-phism pro-jec-tive in-jec-tive re-so-lu-tion ho-mo-lo-gy
  ho-mo-lo-gi-cal ho-mo-lo-gi-cally du-a-liz-ing re-si-due}

%%%%%%%%%%%%%%%%%%%%%%%%%%%%%%%%% VARIA %%%%%%%%%%%%%%%%%%%%%%%%%%%%%%%

\makeatletter \def\@nobreak@#1{\mathchoice%
  {\nobreakdef@\displaystyle\f@size{#1}}%
  {\nobreakdef@\nobreakstyle\tf@size{\firstchoice@false #1}}%
  {\nobreakdef@\nobreakstyle\sf@size{\firstchoice@false #1}}%
  {\nobreakdef@\nobreakstyle\ssf@size{\firstchoice@false #1}}%
  \check@mathfonts}%
\def\nobreakdef@#1#2#3{\hbox{{%
      \everymath{#1}%
      \let\f@size#2\selectfont%
      #3}}}%
\makeatother

\newcommand{\uppar}[1]{{\upshape (}#1\/{\upshape )}}

%% NUMBERING
\numberwithin{equation}{theorem}

% NEW DEFINITIONS
\newcommand{\lw}[2][$^\diamondsuit$]{\nt[#1]{LW:#2}}
\newcommand{\hh}[2][$^\diamondsuit$]{\nt[#1]{HH:#2}}

\newcommand{\RX}{\poly[R]{X}} \newcommand{\RXX}{\pows[R]{X}}
\newcommand{\Ra}{\mspace{-4mu}\Rightarrow\mspace{-4mu}}
\newcommand{\AD}[1]{\operatorname{D}#1}
\newcommand{\ADp}[1]{\operatorname{D}(#1)}
\newcommand{\HH}[2][]{\operatorname{H}_{#1}#2}
\renewcommand{\mod}[1]{\Cat{#1}{mod}}
\newcommand{\Mod}[1]{\Cat{#1}{Mod}} \newcommand{\ab}[1]{b_{#1}}
\newcommand{\crs}[1]{\pmb{#1}} \newcommand{\propp}{\textsc{(p)}}
\newcommand{\propq}{\textsc{(q)}} \newcommand{\tv}{\textsc{(tv)}}
\newcommand{\ac}{\textsc{(ac)}} \newcommand{\gc}{\textsc{(gc)}}
\newcommand{\acp}{\textsc{(ac$'$)}}
\newcommand{\accm}{\textsc{(ac-mcm)}}
\newcommand{\uac}{\textsc{(uac)}}
\newcommand{\uaccm}{\textsc{(uac-mcm)}}
\newcommand{\nc}{\textsc{(nc)}} \newcommand{\gnc}{\textsc{(gnc)}}
\newcommand{\gncm}{\textsc{(arc-g)}} \newcommand{\arc}{\textsc{(arc)}}
\newcommand{\tc}{\textsc{(tc)}} \newcommand{\tci}{\textsc{(tc{\small
      1})}} \newcommand{\tcis}{\textsc{(tc{\footnotesize 1})}}
\newcommand{\tcii}{\textsc{(tc{\small 2})}}
\newcommand{\tciis}{\textsc{(tc{\footnotesize 2})}}
\newcommand{\ltc}{\textsc{(ab\c{s}c)}}
\newcommand{\codimR}{\operatorname{codim}R}
\newcommand{\syz}[3][A]{\operatorname{syz}^#1_#2(#3)}
\newcommand{\Syz}[2][A]{\mathsf{Syz}(#2,#1)}
\newcommand{\Alg}{\Lambda} \newcommand{\Algo}{\Lambda^{\circ}}
\newcommand{\Aop}{A^{\circ}} \newcommand{\Bop}{B^{\circ}}
\newcommand{\x}{\pmb{x}} \newcommand{\X}{\pmb{X}}
\newcommand{\Kc}{\operatorname{K}}

\newenvironment{condition}[2][4em]{\begin{list}{}%
    {\setlength{\leftmargin}{#1}\setlength{\rightmargin}{#2}%
      \setlength{\labelwidth}{3.5em}%
      \setlength{\partopsep}{0pt}%
      \setlength{\topsep}{\thmbotspace}%
      \setlength{\parsep}{0pt}%
      \setlength{\itemsep}{0.3\thmbotspace}} \small} {\end{list}}%

%%% INTRODUCTION

\section*{Introduction}

\noindent The studies of algebras and modules by methods of
homological algebra pivot around cohomology groups and functors---in
particular, their vanishing.  The conjecture of Auslander we refer to
in the abstract asserts that every Artin algebra satisfies the
condition \ac\ defined below.  Auslander's conjecture is stronger than
the Finitistic Dimension Conjecture and several other long-standing
conjectures for finite dimensional algebras---including the
Auslander-Reiten and Nakayama Conjectures; see \cite[ch.~V]{mas1},
\cite{DHp90}, and \cite{KYm96}.  In \cite{DAJLMS04} Jorgensen and
\c{S}ega exhibit a finite dimensional algebra that fails to satisfy
\ac, thereby overturning Auslander's conjecture. This makes relevant a
subtle point: one knows that if \textsl{all} finite dimensional
algebras had satisfied \ac, then they would all have finite finitistic
dimension, but it is \textsl{not} known if a given algebra that
satisfies \ac\ must have finite finitistic dimension. What \textsl{is}
known, is that a finite dimensional algebra $\Alg$ over a field $k$
has finite finitistic dimension if the enveloping algebra $\Alg^e =
\tp[k]{\Alg}{\Algo}$ satisfies \ac; see \seccite[1]{DHp90}.

Such observations motivate the study of \emph{AC rings}, that is,
left-noetherian rings $A$ that satisfy Auslander's condition on
vanishing of cohomology:

\begin{condition}[2.5em]{1.5em}
\item[\ac] For every finitely generated left $A$-module $M$ there
  exists an integer $\ab{M} \ge 0$ such that for every finitely
  generated left $A$-module $N$ one has: $\Ext[A]{\gg 0}{M}{N}=0$
  implies $\Ext[A]{> \ab{M}}{M}{N}=0$.
\end{condition}

\noindent For certain commutative rings this study was initiated by
Huneke and Jorgensen \cite{CHnDAJ03}. In this paper we give special
attention to problems from Auslander's work in representation
theory---including the conjectures mentioned above.

\begin{center}
  $***$
\end{center}

\noindent Auslander and Reiten conjectured \cite{MAsIRt75} that a
finitely generated module $M$ over an Artin algebra $\Alg$ is
projective if $\Ext[\Alg]{i}{M}{M} = 0 = \Ext[\Alg]{i}{M}{\Alg}$ for
all $i \ge 1$. See Appendix~A for a brief survey of this and
related conjectures. To facilitate the discussion, we distinguish
between \textsl{conjectures} (about all algebras) and
\textsl{conditions} (on a single algebra). Consider the following
condition on a left-noetherian ring $A$:

\begin{condition}{4em}
\item[\arc] Every finitely generated left $A$-module $M$ with
  $\Ext[A]{\ge 1}{M}{M\oplus A} = 0$ is projective.
\end{condition}

\noindent The Auslander-Reiten Conjecture can now be restated as ``All
Artin algebras satisfy \arc''.  At the level of conjectures,
Auslander's conjecture is stronger than the Finitistic Dimension
Conjecture, and that one implies the Auslander-Reiten Conjecture.
Thus, had \textsl{all} algebras satisfied \ac, then one would know
that all algebras satisfy \arc. Theorem~A below gives new insight at
the level of conditions: it implies that \textsl{any} given AC~ring
satisfies \arc. Our proof of Theorem~A avoids considerations of
finitistic dimensions, and it remains unknown if every AC~Artin
algebra has finite finitistic dimension.

\begin{spthma}
  Let $A$ be a left-noetherian ring that satisfies \ac, and let $M$ be
  a finitely generated left $A$-module. If one has $\,\Ext[A]{\gg
    0}{M}{M} = 0$ and $\Ext[A]{\ge 1}{M}{A} = 0$, then $M$ is
  projective.
\end{spthma}

\noindent This theorem is a special case of our main
result~\thmref[]{ar}. Notice that the vanishing conditions imposed on
$M$ in Theorem~A appear to be weaker than those in the
Auslander-Reiten Conjecture; we discuss this in \pgref{elab}.

It is an open question---also due to Auslander and Reiten
\cite{MAsIRt91}---whether an Artin algebra is left-Gorenstein if and
only if it is right-Gorenstein. This is known as the Gorenstein
Symmetry Question; the next partial answer is proved in
\prpref[]{artin-id} and~\prpref[]{tachikawa}.  \newpage
\begin{spthmb}
  Let $A$ be a two-sided noetherian ring. If $A$ and $\Aop$ satisfy
  \ac\ and
  \begin{rqm}
  \item $A$ is an Artin algebra, or
  \item $A$ has a dualizing complex {\normalfont (}as defined in\,
    {\normalfont\cite{CFH-06})},
  \end{rqm}
  then $\id[A]{A} < \infty$ if and only if\, $\id[\Aop]{A} < \infty$
  \uppar{whence, $\id[A]{A} = \id[\Aop]{A}$ by {\normalfont
      \cite{YIw80}}}.%
\end{spthmb}%
\noindent We do not know if every Artin algebra has a dualizing
complex, but every finite dimensional $k$-algebra does have one,
cf.~\pgref{dc}.

To study the module category of a Gorenstein ring, Auslander and
Bridger \cite{MAsMBr69} introduced the so-called G-dimension. A
finitely generated left module $M\ne 0$ over a two-sided noetherian
ring $A$ is of G-dimension $0$ if it is reflexive and
$\Ext[A]{i}{M}{A} = 0 = \Ext[\Aop]{i}{\Hom[A]{M}{A}}{A}$ for all $i
\ge 1$. Implicit in their work is the question whether all two-sided
noetherian rings $A$ satisfy the condition:

\begin{condition}[3em]{2em}
\item[\gc] Every finitely generated left $A$-module $M\ne 0$ with
  $\Ext[A]{\ge 1}{M}{A} = 0$ is of G-dimension $0$.
\end{condition}
By another example of Jorgensen and \c{S}ega \cite{DAJLMS06}, also
this question has a negative answer, even for commutative local finite
dimensional $k$-algebras. The following partial answer is part of
\thmref[]{g}.

\begin{spthmc}
  Let $A$ be a two-sided noetherian ring that has a dualizing complex
  {\normalfont (}as defined in {\normalfont \cite{CFH-06})} or is
  commutative.  If $A$ satisfies \ac, then it satisfies \gc.
\end{spthmc}

By work of Huneke, \c{S}ega, and Vraciu \cite{HSV-04}, the
Auslander-Reiten Conjecture holds for commutative noetherian local
rings with radical cube zero, and the counterexamples in
\cite{DAJLMS04,DAJLMS06} show that such rings need not satisfy \ac\ or
\gc. Here is a summary in diagram form:

\vspace*{-4ex}

{\small%
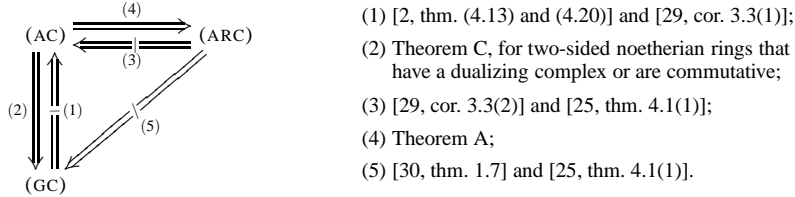
  \begin{equation*}
    \qquad
    \begin{gathered}
      \SelectTips{cm}{} \xymatrix@C=5.5em@R=5.5em{ \ac
        \ar@<1ex>@{=>}[r]^-{(4)} \ar@<-1ex>@{=>}[d]_-{(2)} & \arc
        \ar@<1ex>@{=>}|{\,\mid\,}[l]^-{(3)}
        \ar@{=>}|{\backslash}[ld]^-{(5)}
        \\
        \gc \ar@<-1ex>@{=>}|{-}[u]_-{(1)} }
    \end{gathered}
    \mspace{-110mu}
    \begin{split}
      \text{\footnotesize (1) } & \text{\footnotesize \thmcite[(4.13)
        and
        (4.20)]{MAsMBr69} and \corcite[3.3(1)]{DAJLMS04};}\\
      \text{\footnotesize (2) } & \text{\footnotesize Theorem~C, for
        two-sided noetherian rings that}\\[-1ex] & \text{\footnotesize
        have a dualizing complex or are commutative;}\\
      \text{\footnotesize (3) } & \text{\footnotesize
        \corcite[3.3(2)]{DAJLMS04} and \thmcite[4.1(1)]{HSV-04};}\\
      \text{\footnotesize (4) } & \text{\footnotesize Theorem~A;}\\
      \text{\footnotesize (5) } & \text{\footnotesize
        \thmcite[1.7]{DAJLMS06} and \thmcite[4.1(1)]{HSV-04}.}
    \end{split}
  \end{equation*}
}%

Theorems~A, B, and C are proved in Sections
\secref[]{properties}--\secref[]{Gdim}. In \secref{examples} we
discuss simple procedures for generating new AC rings from existing
ones.

Appendix~A recapitulates certain aspects of the homological
conjectures for finite dimensional $k$-algebras in order to place the
present work in proper perspective.

Theorem~C relies on a technical result, \lemref{nonstandardTEV}, which
owes an intellectual debt to work of Huneke and Jorgensen
\cite{CHnDAJ03}.  Combined with other techniques,
\lemref[]{nonstandardTEV} yields new proofs and modest generalizations
of the main result in \cite{CHnDAJ03} on symmetric Ext-vanishing over
commutative noetherian Gorenstein AC rings; these are given in
Appendix~B.

Many of our proofs use the derived category over a ring. In the next
section we recall the (standard) notation used throughout the paper.

%%%%%%%%%%%%%%%%%%%%%%%%%%%%%%%%%%%%%%%%%%%%%%%%%%%%%%%%%%%%%%%%%%%%%%%%%
\section{Prerequisites}
\label{sec:prerequisites}

Throughout, $A$ denotes a left-noetherian ring which is an algebra
over a commutative ring $\Bbbk$ (e.g.\ \mbox{$\Bbbk = \ZZ$}), and
$\Aop$ is the opposite ring. The letter $k$ denotes a field, and
$\Alg$ denotes a finite dimensional $k$-algebra or, more generally, an
Artin algebra.

\begin{spipg}
  \label{modules}
  Modules (over $A$ or $\Alg$) are left modules, unless otherwise
  specified. We write $\Mod{A}$ for the category of all $A$-modules
  and $\mod{A}$ for the full subcategory of finitely generated
  $A$-modules.

  For $M$ and $N$ in $\Mod{A}$, the notation $\Ext[A]{\ge n}{M}{N} =0$
  means that $\Ext[A]{i}{M}{N}$ vanish for all $i\ge n$. We write
  $\Ext[A]{\gg 0}{M}{N} =0$ if $\Ext[A]{\ge n}{M}{N} =0$ for some~$n$.
  For $M$ in $\mod{A}$, a number $\ab{M}$ with the property required
  in \ac, see the Introduction, is called an \emph{Auslander bound}
  for $M$. We also consider rings $A$ over which there is a uniform
  Auslander bound for all $M$ in $\mod{A}$, i.e.~rings that satisfy:

  \begin{condition}{4em}
  \item[\uac] There is a $b \ge 0$ such that for all finitely
    generated $A$-modules $M$ and $N$ one has: \\ $\Ext[A]{\gg
      0}{M}{N}=0$ implies $\Ext[A]{> b}{M}{N}=0$.
  \end{condition}

  In \cite{CHnDAJ03} the smallest integer $b$ with this property is
  called the Ext-index of $A$.
\end{spipg}

\begin{spipg}
  A complex of $A$-modules is graded homologically,
  \begin{equation*}
    M \deq \cdots \lra M_{\d+1} \dra{\d+1}{M} M_{\d} \dra{\d}{M}
    M_{\d-1} \lra \cdots,
  \end{equation*}
  and, for short, called an \emph{$A$-complex.} The suspension of $M$
  is the complex $\Shift{M}$ with $\Shiftp{M}_\d = M_{\d-1}$ and
  $\dif{\Shift{M}} = - \dif{M}$. With the notation
  \begin{equation*}
    \Co[\d]{M} = \Coker{\dif[\d+1]{M}} \qqand \Cy[\d]{M} =
    \Ker{\dif[\d]{M}},
  \end{equation*}
  soft truncations of $M$ are defined as
  \begin{align*}
    \Tsa{u}{M} &\deq 0 \to \Co[u]{M} \to M_{u-1} \to M_{u-2} \to
    \cdots
    \text{ and} \\
    \Tsb{w}{M} &\deq \cdots \to M_{w+2} \to M_{w+1} \to \Cy[w]{M} \to
    0.
  \end{align*}
  The hard truncations of $M$ are defined as
  \begin{align*}
    \Tha{u}{M} \deq 0 \to M_{u} \to M_{u-1} \to \cdots \qand
    \Thb{w}{M} \deq \cdots \to M_{w+1} \to M_w \to 0.
  \end{align*}
  We say that $M$ is \emph{left-bounded} if $M_\d =0$ for $\d \gg 0$,
  \emph{right-bounded} if $M_\d =0$ for $\d \ll 0$, and \emph{bounded}
  if $M_\d = 0$ for $|\d| \gg 0$.  If the homology complex $\H{M}$ is
  (left/right-) bounded, then $M$ is said to be \emph{homologically
    \uppar{left/right-}bounded}. The notation $\sup{M}$ and $\inf{M}$
  is used for the supremum and infimum of the set
  $\setof{\dinZ}{\H[\d]{M}\ne 0}$ with the conventions that
  $\sup\varnothing = -\infty$ and $\inf\varnothing = \infty$.

  A morphism $\a$ of complexes is called a \emph{quasiisomorphism},
  and marked by the symbol $\qis$, if it induces an isomorphism on the
  level of homology.  The mapping cone of $\a$ is denoted
  $\Cone{\a}$. Recall that the complex $\Cone{\a}$ is exact if and
  only if $\a$ is a quasiisomorphism.  Quasiisomorphisms between
  $A$-complexes are isomorphisms in the derived category
  $\D[A]$. Isomorphisms in $\D[A]$ are also marked by the
  symbol~$\qis$.
\end{spipg}

\begin{spipg}
  We use standard notation, $\RHom[A]{-}{-}$ and $\Ltp[A]{-}{-}$, for
  the right derived Hom functor and the left derived tensor product
  functor; see \cite[ch.~10]{wei}. Recall that for all $A$-modules $M$
  and $N$ and all $\Aop$-modules $K$ there are isomorphisms
  \begin{equation*}
    \Ext[A]{i}{M}{N} \is \HH[-i]{\RHom[A]{M}{N}} \qand \Tor[A]{i}{K}{M} \is
    \H[i]{\Ltp[A]{K}{M}}.
  \end{equation*}

  Resolutions of complexes, projective dimension (pd), and injective
  dimension (id) are treated in \cite{LLAHBF91}. We make frequent use
  of the following: Every homologically left-bounded complex has a
  left-bounded injective resolution; every homologically right-bounded
  complex $M$ has a right-bounded free resolution $L$, and if $M$ has
  degreewise finitely generated homology, then $L$ can be taken to be
  degreewise finitely generated. In particular, every homologically
  right-bounded complex $M$ has a projective resolution and the
  projective dimension is given as:
  \begin{equation*}
    \pd{M} = \inf \left\{ \, \sup \{ i \in \ZZ \mid P_{i} \ne 0 \} \,
      \left|\;
        \mbox{$P$ is a projective resolution of $M$}
      \right.
    \right\}.
  \end{equation*}
  The injective dimension of a complex is defined similarly.
\end{spipg}

\begin{splem}
  \label{lem:Ext}
  Let $X$ and $Y$ be $A$-complexes. Assume that $X$ is
  homologically~right-bounded and let \mbox{$P \qra X$} be a
  projective resolution; assume that $Y$ is homologically left-bounded
  and let \mbox{$Y \qra I$} be an injective resolution.  If
  $\,\RHom[A]{X}{Y}$ is homologically bounded and $s \ge \sup{X}$,
  then $\Ext[A]{\ge 1}{\Co[s]{P}}{\Cy[\d]{I}}=0$ for all $0 \gg \d$.
\end{splem}

\begin{proof}
  Let $s\ge \sup{X}$ and note that $\Tsa{s}{P} \qis X$ in $\D[A]$.
  Application of $\RHom[A]{-}{Y}$ to the distinguished triangle in
  $\D[A]$,
  \begin{equation*}
    \Tha{s-1}{P} \lra \Tsa{s}{P} \lra \Shift[s]{\Co[s]{P}} \lra,
  \end{equation*}
  induces a long exact sequence of homology modules, which yields
  isomorphisms
  \begin{align*}
    \HH[\d+1]{\RHom[A]{\Tha{s-1}{P}}{Y}} %%
    &\is \HH[\d]{\RHom[A]{\Shift[s]{\Co[s]{P}}}{Y}} \\
    &\is \HH[\d+s]{\RHom[A]{\Co[s]{P}}{Y}},
  \end{align*}
  for $\d+1 < \inf{\RHom[A]{X}{Y}}$. Obviously,
  $\pd[A]{(\Tha{s-1}{P})} \le s-1$ and, therefore,
  \begin{equation*}
    \inf{\RHom[A]{\Tha{s-1}{P}}{Y}} \ge \inf{Y} -(s-1);
  \end{equation*}
  see \thmcite[2.4.P]{LLAHBF91}. Set \mbox{$w =
    \min\set{\inf{Y},\inf{\RHom[A]{X}{Y}}+s-1}$}; it follows that
  \begin{equation}
    \label{eq:a1}
    \HH[\d]{\RHom[A]{\Co[s]{P}}{Y}} =0 \text{ for all }
    \d < w. 
  \end{equation}
  If $\d \le w$, then $\d \le \inf{Y}$, so there is an isomorphism of
  module functors
  \begin{equation*}
    \Ext[A]{i}{-}{\Cy[\d]{I}} \is \HH[\d-i]{\RHom[A]{-}{Y}}
  \end{equation*}
  for every $i >0$, cf.~\cite[proof of lem.~(6.1.12)]{lnm}. In
  particular,
  \begin{equation*}
    \Ext[A]{i}{\Co[s]{P}}{\Cy[\d]{I}} \is
    \HH[\d-i]{\RHom[A]{\Co[s]{P}}{Y}} = 0 
  \end{equation*}
  for all $i>0$, where the last equality follows from \eqref{a1}. \qed
\end{proof}

%%%%%%%%%%%%%%%%%%%%%%%%%%%%%%%%%%%%%%%%%%%%%%%%%%%%%%%%%%%%%%%%%%%%%

\section{The Auslander-Reiten Conjecture}
\label{sec:properties}

In this section we prove Theorem~A from the Introduction.  We open
with a technical lemma.

\begin{splem}
  \label{lem:exact}
  Assume that $A$ satisfies \ac.  Let $U$ be an exact $A$-complex and
  $C$ be a finitely generated $A$-module. If
  \begin{prt}
  \item $U_\d$ is finitely generated for all $\d \gg 0$,
  \item $\Ext[A]{\ge 1}{C}{U_\d}=0$ for all $\d \in \ZZ$, and
  \item there exists a $w \in \ZZ$ such that $\Ext[A]{\gg
      0}{C}{\Cy[w]{U}}=0$,
  \end{prt}
  then $\Ext[A]{\ge 1}{C}{\Cy[\d]{U}}=0$ for all $\d \in \ZZ$. In
  particular, $\Hom[A]{C}{U}$ is exact.
\end{splem}

\begin{proof}
  Apply $\Hom[A]{C}{-}$ to $0 \to \Cy[\d+1]{U} \to U_{\d+1} \to
  \Cy[\d]{U} \to 0$, then (b) yields%
  \begin{equation}
    \label{eq:one}
    \Ext[A]{i}{C}{\Cy[\d]{U}} \is \Ext[A]{i+n}{C}{\Cy[\d+n]{U}} 
    \text{ for all $\d \in \ZZ$, $i>0$, and $n\ge 0$}.
  \end{equation}
  If $\d \ge w$, then $\Ext[A]{\gg 0}{C}{\Cy[\d]{U}}=0$. Indeed,
  \eqref{one} yields isomorphisms
  \begin{align*}
    \Ext[A]{i+\d-w}{C}{\Cy[\d]{U}} \dis
    \Ext[A]{i+(\d-w)}{C}{\Cy[w+(\d-w)]{U}} \dis
    \Ext[A]{i}{C}{\Cy[w]{U}},
  \end{align*}
  for $i>0$, and the right-most Ext group vanishes by (c) for $i \gg
  0$.  By (a) there is an integer $t$ such that $\Cy[\d]{U}$ is
  finitely generated for $\d \ge t$. As $A$ satisfies \ac,
  \begin{equation}
    \label{eq:two}
    \Ext[A]{>\ab{}}{C}{\Cy[\d]{U}}=0 \text{ for all } \d
    \ge m=\max\set[]{t,w}, 
  \end{equation}
  where $\ab{}$ is an Auslander bound for $C$. To see that
  $\Ext[A]{\ge 1}{C}{\Cy[\d]{U}}=0$ for all $\d$, consider the cases
  $\d \ge m-\ab{}$ and $\d \le m-\ab{}$ separately. In the following,
  let $i>0$. If $\d \ge m-\ab{}$, then
  \begin{equation*}
    \Ext[A]{i}{C}{\Cy[\d]{U}} \is \Ext[A]{i+\ab{}}{C}{\Cy[\d+\ab{}]{U}} = 0
  \end{equation*}
  by \eqref{one} and \eqref{two}. If $\d \le m-\ab{}$ then, in
  particular, $m-\d \ge \ab{} \ge 0$, and thus
  \begin{displaymath}
    \Ext[A]{i}{C}{\Cy[\d]{U}} \is \Ext[A]{i+(m-\d)}{C}{\Cy[\d+(m-\d)]{U}} =
    \Ext[A]{i+m-\d}{C}{\Cy[m]{U}} = 0;
  \end{displaymath}
  again by \eqref{one} and \eqref{two}. \qed
\end{proof}

\begin{sprmk}
  The lemma above may fail for rings that do not satisfy \ac.  Indeed,
  one counterexample to Auslander's conjecture is a commutative local
  self-injective finite dimensional $k$-algebra $R$ for which there
  exist finitely generated modules $C$ and $Z$, such that
  $\Ext{i}{C}{Z} \ne 0$ if and only if $i= 0,1$; see
  \corcite[3.3.(1)]{DAJLMS04}. Because $R$ is self-injective, the
  modules $C$ and $Z$ have G-dimension $0$; see
  \prpcite[3.8]{MAsMBr69}. Let $U$ be a complete projective resolution
  of $Z$, see \thmcite[(4.1.4)]{lnm}, then $U$ and $C$ satisfy the
  requirements in \lemref{exact}, but $\Ext{1}{C}{Z} \ne 0$.
\end{sprmk}

Theorem~A in the Introduction is an immediate consequence of the next
result.

\begin{spthm}
  \label{thm:ar}
  Assume that $A$ satisfies \ac, and let $M$ be an $A$-complex. If $M$
  has bounded and degreewise finitely generated homology, and
  $\RHom[A]{M}{M \oplus A}$ is homologically bounded, then $M$ has
  finite projective dimension given by
  \begin{displaymath}
    \pd[A]{M} = -\inf{\RHom[A]{M}{A}} < \infty.
  \end{displaymath}
\end{spthm}

\begin{proof}
  We may assume that $M \not\qis 0$ in $\D[A]$. We need only prove
  that $\pd[A]{M}$ is finite, then a standard argument yields the
  equality displayed above; see the proof of
  \prpcite[(2.3.10)]{lnm}. Take a right-bounded resolution $L \qra M$
  by finitely generated free $A$-modules and consider the integer
  \begin{equation*}
    s = \max\set{-\inf{\RHom[A]{M}{A}},\, \sup{M}}.
  \end{equation*}
  We will show that the cokernel $\Co[s]{L}$ is projective, i.e.\
  $\Ext[A]{1}{\Co[s]{L}}{\Co[s+1]{L}}=0$. To this end, take an
  injective resolution $M \qra I$ with $I_\d =0$ for $\d > \sup{M}$;
  see \corcite[2.7.I]{LLAHBF91}.  Since $\RHom[A]{M}{M}$ is
  homologically bounded, there is by \lemref{Ext} an integer $u \le
  \inf{M}$ such that
  \begin{equation}
    \label{eq:b1}
    \Ext[A]{\ge 1}{\Co[s]{L}}{\Cy[u]{I}} = 0.
  \end{equation}
  There are quasiisomorphisms
  \begin{displaymath}
    L \qra M \qra I \qla \Tsb{u}{I},
  \end{displaymath}
  so by \cite[1.4.P]{LLAHBF91} there is a quasiisomorphism $\alpha
  \colon L \qra \Tsb{u}{I}$.  We claim that \lemref{exact} applies to
  $U=\Cone{\alpha}$ and the finitely generated module $C = \Co[s]{L}$.
  Requirement \partlemref[]{exact}{a} is clearly met, and so is
  \partlemref[]{exact}{c}, as $\Cone{\a}$ is right-bounded.  To verify
  \partlemref[]{exact}{b} it suffices, in view of \eqref{b1}, to show
  that $\Ext[A]{\ge 1}{\Co[s]{L}}{A}=0$, and this follows as
  \begin{equation*}
    \Ext[A]{i}{\Co[s]{L}}{A} \is
    \HH[-(i+s)]{\RHom[A]{M}{A}} = 0 \text{ for all $i>0$;}
  \end{equation*}
  cf.~\cite[prf.\ of (4.3.9)]{lnm}. In particular, \lemref{exact}
  gives $\Ext[A]{\ge 1}{\Co[s]{L}}{\Cy[s+1]{\Cone{\a}}}=0,$ and by the
  choice of $I$ we have $\Cy[s+1]{\Cone{\a}} = \Co[s+1]{L}$. \qed
\end{proof}

\begin{sprmk}
  \label{elab}
  The condition \arc\ and Theorem A in the Introduction draw identical
  conclusions from apparently different assumptions on a finitely
  generated $A$-module $M$, namely:
  \begin{prt}
  \item $\Ext[A]{\ge 1}{M}{M\oplus A} = 0$; compared to
  \item $\Ext[A]{\gg 0}{M}{M} = 0$ and $\Ext[A]{\ge 1}{M}{A} = 0$.
  \end{prt}
  Clearly, (a) implies (b). We do not know if the two are equivalent,
  not even if $A$ is commutative local and Gorenstein.  \thmref{ar}
  shows that if $A$ is AC, then (a) and (b) are equivalent. A much
  stronger result holds if $A$ is commutative local and complete
  intersection, then $\Ext[A]{\ge 1}{M}{A} = 0$ and vanishing of
  $\Ext[A]{2i}{M}{M}$ for a single integer $i>0$ implies that $M$ is
  free; see~\thmcite[4.2]{LLAROB00}. If $A$ is commutative local (AC
  or not) with radical cube zero, then vanishing of
  $\Ext[A]{i}{M}{M\oplus A}$ for four consecutive values of $i\ge 2$
  implies that $M$ is free; see \thmcite[4.1]{HSV-04}.
\end{sprmk}

\section{The Gorenstein Symmetry Question}
\label{sec:gs}

For a two-sided noetherian ring $A$, we do not know if Auslander's
condition is symmetric, that is, if $A$ and $\Aop$ satisfy \ac\
simultaneously. For Artin algebras, however, the uniform condition
\uac, defined in \pgref{modules}, is symmetric.

\begin{spobs}
  \label{obs:artin}
  Let $\Alg$ be an Artin algebra. The canonical duality functor
  $$\AD{}\colon \mod{\Algo} \lra \mod{\Alg},$$ see
  \thmcite[II.3.3]{rta}, provides isomorphisms
  \begin{equation*}
    \Ext[\Algo]{i}{M}{N} \is
    \Ext[\Alg]{i}{\ADp{N}}{\ADp{M}}
  \end{equation*}
  for all finitely generated $\Algo$-modules $M$ and $N$ and all
  integers $i$.  This shows that $\Algo$ satisfies \uac\ if and only
  if $\Alg$ does.
\end{spobs}

Auslander and Reiten \cite{MAsIRt91} raise the question whether an
Artin algebra is left-Gorenstein if and only if it is
right-Gorenstein. The next proposition contains part $(1)$ of
Theorem~B from the Introduction, and it uses \prpcite[6.10]{MAsIRt91}
to establish an ``algebra-wise'' relation between Auslander's
conjecture and the Finitistic Dimension Conjecture.

\begin{spprp}
  \label{prp:artin-id}
  Let $\Alg$ be an Artin algebra that satisfies \ac.  If\,
  $\id[\Alg]{\Alg}$ is finite, then $\id[\Algo]{\Alg}$ and the
  finitistic dimension of $\Alg$ \uppar{on both sides}\footnote{In
    general, it is not known if the left-finitistic dimension of a
    finite dimensional algebra is finite if the right-finitistic
    dimension is, but one knows that they may differ; see
    \cite[exa.~2.2]{CUJHLn82}.} is finite.
\end{spprp}

\begin{proof}
  The finitely generated $\Alg$-module $\ADp{\Alg_\Alg}$ is injective.
  Set $n=\id[\Alg]{\Alg}$, then
  \begin{equation*}
    \Ext[\Alg]{>n}{\ADp{\Alg_\Alg}}{\ADp{\Alg_\Alg} \oplus
      {}_\Alg\Alg}=0,
  \end{equation*}
  so it follows from \thmref{ar} that $\pd[\Alg]{\ADp{\Alg_\Alg}}$ is
  at most $n$. For every finitely generated $\Algo$-module $N$, the
  isomorphism from \obsref{artin} yields
  \begin{align*}
    \Ext[\Algo]{i}{N_\Alg}{\Alg_\Alg} \is
    \Ext[\Alg]{i}{\ADp{\Alg_\Alg}}{\ADp{N_\Alg}} = 0 \text{ for $i >
      n$},
  \end{align*}
  whence $\id[\Algo]{\Alg} \le n$. Now the finitistic dimension of
  $\Alg$ is finite by \prpcite[6.10]{MAsIRt91}. \qed
\end{proof}

\begin{sprmk}
  For an Artin algebra $\Alg$ that satisfies \uac, it follows from
  \obsref{artin} and \prpref{artin-id} that $\id[\Alg]{\Alg}$ is
  finite if and only if $\id[\Algo]{\Alg}$ is finite.

  Nagata's regular ring of infinite Krull dimension \cite[ex.~1,
  p.~203]{Nag} is an example of a commutative noetherian ring that
  satisfies \ac\ but not \uac. However, in the realm of Artin algebras
  (or local rings) we do not know of such an example.
\end{sprmk}

Part $(2)$ of Theorem~B is a special case of \prpref{tachikawa} below,
which addresses a natural generalization of the conditions \tci\ and
\ltc\ discussed in Appendix~A.

\newcommand{\bi}[1]{{}_A#1_B}
\newcommand{\lhty}[1][D]{\grave{\chi}_#1^{\langle A,B \rangle}}
\newcommand{\rhty}[1][D]{\acute{\chi}_#1^{\langle A,B \rangle}}

\begin{spipg}
  \label{dc}
  Let $B$ be a right-noetherian ring, which is also a $\Bbbk$-algebra;
  \prpref{tachikawa} involves a dualizing complex ${}_AD_B$ for the
  pair $\langle A,B\rangle$ in the sense of \cite[def.~1.1]{CFH-06}.
  That~is,
  \begin{rqm}
  \item The complex $D$ has bounded and degreewise finitely generated
    homology over $A$ and over $\Bop$.
  \item There exists a quasi-isomorphism of complexes of bimodules,
    $\bi{P} \qra \bi{D}$, where $\bi{P}$ is right-bounded and consists
    of modules that are projective over $A$ and over $\Bop$.
  \item There exists a quasi-isomorphism of complexes of bimodules,
    $\bi{D} \qra \bi{I}$, where $\bi{I}$ is bounded and consists of
    modules that are injective over $A$ and over~$\Bop$.
  \item The homothety morphisms
    \begin{equation*}
      \qquad _AA_A \longrightarrow
      \RHom[\Bop]{\bi{D}}{\bi{D}}\quad \text{and}\quad
      _BB_B \longrightarrow \RHom[A]{\bi{D}}{\bi{D}},
    \end{equation*}
    are isomorphisms in homology.
  \end{rqm}
  If $A$ is two-sided noetherian, then a dualizing complex for
  $\langle A,A\rangle$ is called a dualizing complex for $A$. This
  generalizes the definition for commutative rings in
  \cite[V.\S2]{rad}.
  
  We do not know if every Artin $\Bbbk$-algebra $\Alg$ has a dualizing
  complex. To be precise, we do not know if the obvious candidate
  $D=\Hom[\Bbbk]{{}_\Alg\Alg_\Alg}{\Bbbk}$ has a resolution by
  $\Alg$-bimodules, as required in $(2)$. If $\Bbbk$ is a field,
  however, this $D$ is a dualizing complex for $\Alg$; see
  \cite[exa.~2.3(b)]{AYkJJZ99} and \cite[app.~A]{CFH-06}.
\end{spipg}

\begin{spprp}
  \label{prp:tachikawa}
  Let the rings $A$ and $B$ be as in {\rm \pgref{dc}}, and let $D$ be
  a dualizing complex for the pair $\langle A,B\rangle$. The complexes
  $\RHom[A]{D}{A}$ and $\RHom[\Bop]{D}{B}$ are isomorphic in
  $\D[\Bbbk]$, and when they are homologically bounded, the following
  hold:
  \begin{prt}
  \item If $A$ satisfies \ac, then $\id[\Aop]{A}$ and $\id[\Bop]{B}$
    are at most $\pd[A]{D} + \id[\Bop]{D} < \infty$.
  \item If $\Bop$ satisfies \ac, then $\id[A]{A}$ and $\id[B]{B}$ are
    at most $\pd[\Bop]{D} + \id[A]{D} < \infty$.
  \end{prt}
\end{spprp}

\begin{proof}
  The first assertion is an elementary application of swap in
  $\D[\Bbbk]$:
  \begin{align*}
    \RHom[A]{{}_AD}{{}_AA} %%
    & \qis \RHom[A]{{}_AD}{\RHom[\Bop]{D_B}{{}_AD_B}} \\
    & \qis \RHom[\Bop]{D_B}{\RHom[A]{{}_AD}{{}_AD_B}} \\
    & \qis \RHom[\Bop]{D_B}{B_B}.
  \end{align*}
  
  By symmetry it suffices to prove part (a). As $\RHom[A]{D}{A}$ is
  homologically boun\-ded, it follows from \thmref{ar} that
  $\pd[A]{D}$ is finite. For every $\Aop$-module $M$ we have
  \begin{align*}
    -\inf{\RHom[\Aop]{M_A}{A_A}} %%
    &= -\inf{\RHom[\Aop]{M_A}{\RHom[\Bop]{{}_AD_B}{D_B}}} \\
    &= -\inf{\RHom[\Bop]{\Ltp[A]{M_A}{{}_AD_B}}{D_B}} \\
    &\le \id[\Bop]{D} + \sup{\Ltpp[A]{M_A}{{}_AD_B}} \\
    &\le \id[\Bop]{D} + \pd[A]{D},
  \end{align*}
  where the inequalities are by \thmcite[2.4.I and 2.4.F]{LLAHBF91}.
  Thus, $\id[\Aop]{A}$ is at most $\id[\Bop]{D} + \pd[A]{D}$ by
  \thmcite[2.4.I]{LLAHBF91}.  Similarly, for every $\Bop$-module $N$
  we have
  \begin{align*}
    -\inf{\RHom[\Bop]{N_B}{B_B}} %%
    &= -\inf{\RHom[\Bop]{N_B}{\RHom[A]{{}_AD}{{}_AD_B}}} \\
    &= -\inf{\RHom[A]{{}_AD}{\RHom[\Bop]{N_B}{{}_AD_B}}} \\
    &\le \pd[A]{D} - \inf{\RHom[\Bop]{N_B}{{}_AD_B}} \\
    &\le \pd[A]{D} + \id[\Bop]{D};
  \end{align*}
  this time by \thmcite[2.4.P and 2.4.I]{LLAHBF91}. \qed
\end{proof}

%%%%%%%%%%%%%%%%%%%%%%%%%%%%%%%%%%%%%%%%%%%%%%%%%%%%%%%%%%%%%%%%%%%%%%%%%

\section{Functoriality of G-dimension}
\label{sec:Gdim}

Now we prove Theorem~C from the Introduction; our proof hinges on the
following lemma about invertibility of the tensor evaluation morphism;
cf.~\cite[4.3]{LLAHBF91}.

\begin{splem}
  \label{lem:nonstandardTEV}
  Let $M$ and $N$ be $A$-complexes and $T$ be an $A$-bimodule. Assume
  that $M$ and $N$ have bounded and degreewise finitely generated
  homology and that ${}_AT$ is finitely generated. Consider the tensor
  evaluation morphism in $\D[\Bbbk]$:
  \begin{equation*}
    \dmapdef{\tev{MTN}}{\Ltp[A]{\RHom[A]{M}{T}}{N}}%
    {\RHom[A]{M}{\Ltp[A]{T}{N}}}.
  \end{equation*}
  If $A$ satisfies \ac\ and the three complexes
  \begin{equation*}
    \RHom[A]{M}{T},\quad \Ltp[A]{T}{N}, \qand \RHom[A]{M}{\Ltp[A]{T}{N}}
  \end{equation*}
  are homologically bounded, then $\tev{MTN}$ is an isomorphism.
\end{splem}

The lemma may fail if $A$ does not satisfy \ac; see remarks after the
proof.

\begin{proof}
  Take right-bounded resolutions \mbox{$P \qra M$} and \mbox{$Q \qra
    N$} by finitely generated free $A$-modules. The goal is to prove
  that the tensor evaluation morphism $\tev{PTQ}$ is a
  quasiisomorphism in the category of $\Bbbk$-complexes. This is
  achieved as follows: As $\Ltp[A]{T}{N}$ is homologically bounded, we
  may take a left-bounded injective resolution
  $\qisdef{\rho}{\tp[A]{T}{Q}}{I}$. Set $s =
  \max\set{\sup{M},-\inf{\RHom[A]{M}{T}}}$; it is an integer as we are
  free to assume $M\not\qis 0$ in $\D[A]$. There is now a
  quasiisomorphism $\qisdef{\tau}{P}{\Tsa{s}{P}}$. Consider the
  commutative diagram in the category of $\Bbbk$-complexes
  \begin{equation*}
    \xymatrix@C=6em{
      \tp[A]{\Hom[A]{P}{T}}{Q}
      \ar[r]^-{\tev{PTQ}}
      & \Hom[A]{P}{\tp[A]{T}{Q}}
      \ar[d]_-{\qis}^{\Hom[A]{P}{\rho}}
      \\
      \tp[A]{\Hom[A]{\Tsa{s}{P}}{T}}{Q}
      \ar[u]^-{\tp[A]{\Hom[A]{\tau}{T}}{Q}}
      \ar[d]^-\is_{\tev{\Tsa{s}{P}TQ}}
      & \Hom[A]{P}{I}
      \\
      \Hom[A]{\Tsa{s}{P}}{\tp[A]{T}{Q}}
      \ar[r]^-{\Hom[A]{\Tsa{s}{P}}{\rho}}
      & \Hom[A]{\Tsa{s}{P}}{I}.
      \ar[u]^-{\qis}_-{\Hom[A]{\tau}{I}}
    }
  \end{equation*}
  The vertical morphisms on the right are clearly quasiisomorphisms,
  and the tensor evaluation morphism $\tev{\Tsa{s}{P}TQ}$ is easily
  seen to be invertible, cf.~\prpcite[2.1(v)]{LWCHHlb}. It remains to
  prove that $\tp[A]{\Hom[A]{\tau}{T}}{Q}$ and
  $\Hom[A]{\Tsa{s}{P}}{\rho}$ are quasiisomorphisms.
  
  For the first one, it is sufficient to demonstrate exactness of
  \begin{displaymath}
    \Cone{\Hom[A]{\tau}{T}} \is \Shift{\Hom[A]{\Cone{\tau}}{T}}.
  \end{displaymath}
  Since the complex $\Cone{\tau}$ is exact and right-bounded, it is
  enough to argue that
  \begin{displaymath}
    \Ext[A]{\ge 1}{(\Cone{\tau})_\d}{T}=0 \text{ for all }
    \d \in \ZZ.
  \end{displaymath}
  For $\d \ne s$ this is clear, as the module $\Conep{\tau}_\d$ is
  projective. Since $\Conep{\tau}_s = \Co[s]{P}\oplus P_{s-1}$, the
  case $\d=s$ follows from the isomorphisms
  \begin{equation}
    \label{eq:ono}
    \Ext[A]{i}{\Co[s]{P}}{T} \is
    \HH[-(s+i)]{\RHom[A]{M}{T}} = 0 \text{ for all } i>0,
  \end{equation}  
  which are immediate by the choice of $s$; cf.~\cite[proof of
  lem.~(4.3.9)]{lnm}.
  
  To see that $\Hom[A]{\Tsa{s}{P}}{\rho}$ is a quasiisomorphism, it
  suffices by \prpcite[2.6(a)]{CFH-06} to argue that
  $\Hom[A]{\Tsap{s}{P}_\d}{\rho}$ is a quasiisomorphism for all $\d
  \in \ZZ$. For $\d \ne s$ this is clear, as $\Tsap{s}{P}_\d$ is
  projective. Since $\Tsap{s}{P}_s = \Co[s]{P}$, the case $\d=s$ is
  equivalent to exactness of
  \begin{displaymath}
    \Cone{\Hom[A]{\Co[s]{P}}{\rho}} \is \Hom[A]{\Co[s]{P}}{\Cone{\rho}}.
  \end{displaymath}
  To complete the proof we show that \lemref{exact} applies to the
  complex $\Cone{\rho}$ and the finitely generated module $\Co[s]{P}$.
  Since $\Conep{\rho}_\d = I_\d \oplus \tpp[A]{T}{Q}_{\d-1}$, where
  $\tpp[A]{T}{Q}_{\d-1}$ is a finite direct sum of copies of ${}_AT$,
  it follows from \eqref{ono} that requirement \lemref[]{exact}(b) is
  fulfilled.  Furthermore, since $I_\d =0$ for $\d \gg 0$ also
  \lemref[]{exact}(a) is met. Finally, homological boundedness of
  $\RHom[A]{M}{\Ltp[A]{T}{N}}$ implies by \lemref{Ext} that
  $\Ext[A]{\ge 1}{\Co[s]{P}}{\Cy[\d]{I}} = 0$ for all $\d \ll 0$.
  Since $\Cy[\d]{\Cone{\rho}} = \Cy[\d]{I}$ for $\d \ll 0$, also
  requirement \lemref[]{exact}(c) is fulfilled. \qed
\end{proof}

In \cite{DAJLMS04} is given an example of a self-injective finite
dimensional $k$-algebra that does not satisfy \ac, so it follows from
the next proposition that \lemref{nonstandardTEV} may fail for a ring
that does not satisfy \ac.

For a Gorenstein ring---i.e.\ a two-sided noetherian ring with
$\id[A]{A}$ and $\id[\Aop]{A}$ finite---the equivalence of \eqclbl{i}
and \eqclbl{ii} below is proved by Mori \thmcite[3.3]{IMr07}.

\begin{spprp}
  \label{prp:equivalentAB}
  If\, $\id[A]{A}$ is finite, then the following conditions are
  equivalent:
  \begin{eqc}
  \item $A$ satisfies \ac.
  \item $A$ satisfies \uac.
  \item For all $A$-complexes $M$ and $N$ with bounded and degreewise
    finitely generated homology one has: if\, $\RHom[A]{M}{N}$ is
    homologically bounded, then
    \begin{equation*}
      \dmapdef{\tev{MAN}}{\Ltp[A]{\RHom[A]{M}{A}}{N}}{\RHom[A]{M}{N}}
    \end{equation*}    
    is an isomorphism in $\D[\Bbbk]$.
  \end{eqc}
\end{spprp}

\begin{proof}
  Since $\id[A]{A}$ is finite, the implication $(i)\Ra(iii)$ follows
  by \lemref{nonstandardTEV}. Obviously $(ii)$ implies $(i)$, so it
  remains to show the implication $(iii)\Ra(ii)$.
  
  Let $M$ and $N$ be finitely generated $A$-modules such that
  $\Ext[A]{\gg 0}{M}{N}=0$. This means that $\RHom[A]{M}{N}$ is
  bounded, so by $(iii)$ there is an isomorphism
  \begin{equation*}
    \Ltp[A]{\RHom[A]{M}{A}}{N} \qra \RHom[A]{M}{N}
  \end{equation*}  
  in $\D[\Bbbk]$. Consequently,
  \begin{align*}
    -\inf{\RHom[A]{M}{N}} %%
    &= -\inf{\Ltpp[A]{\RHom[A]{M}{A}}{N}} \\
    &\le -\inf{\RHom[A]{M}{A}} \\
    &\le \id[A]{A},
  \end{align*}
  where the first inequality follows by \lemcite[2.1.(2)]{HBF77b} and
  the second by \cite[2.4.I]{LLAHBF91}. This shows that
  $\Ext[A]{i}{M}{N}=0$ for all $i > \id[A]{A}$. \qed
\end{proof}

% \pagebreak

\begin{spipg}
  One says that the G-dimension is functorial over a two-sided
  noetherian ring if it satisfies the condition \gc\ from the
  Introduction. Examples of such rings include:
  \begin{itemlist}
  \item Gorenstein rings; see \prpcite[(3.8)]{MAsMBr69}.
  \item Commutative noetherian rings that are locally Gorenstein,
    see~\cite[(1.3.2)]{lnm}.
  \item Local Artin algebras with radical square zero; see
    \prpcite[2]{MSM74}.
  \item Commutative noetherian Golod local rings,
    see~\prpcite[1.4]{DAJLMS06}.
  \end{itemlist}
  \noindent The next result establishes Theorem~C from the
  Introduction, which adds (certain) AC rings to the list above.
\end{spipg}

\begin{theorem}
  \label{thm:g}
  Let $A$ be a two-sided noetherian ring that satisfies \ac, and
  assume that $A$ has a dualizing complex or is commutative. For every
  $A$-complex $M$ with bounded and degreewise finitely generated
  homology there is an equality:
  \begin{equation*}
    \Gdim[A]{M} = -\inf{\RHom[A]{M}{A}}.
  \end{equation*}
\end{theorem}

\begin{sprmk}
  Jorgensen and \c{S}ega \cite{DAJLMS06} construct a commutative local
  finite dimensional $k$-algebra $R$ and a finitely generated
  $R$-module $M$ with $\Ext{\ge 1}{M}{R}=0$ but infinite G-dimension.
  Note that in view of \thmref{g}, $R$ cannot satisfy \ac.  Further,
  it has length $8$ and $\m^3=0$, where $\m$ is its radical, and thus
  this example is minimal: Primarily with respect to the invariant
  $\min\setof[]{n}{\m^n=0}$---as every ring with radical square zero
  satisfies \ac\ by \prpcite[1.1]{DAJLMS04}. Secondarily with respect
  to length---as every commutative local artinian ring with radical
  cube zero and length at most $7$ satisfies \ac, also by
  \prpcite[1.1]{DAJLMS04}.
\end{sprmk}

\begin{pfofa}
  First assume that $A$ is commutative. It is sufficient to prove that
  homological boundedness of $\RHom[A]{M}{A}$ implies that the
  biduality morphism
  \begin{equation*}
    \dmapdef{\delta_M^A}{M}{\RHom[A]{\RHom[A]{M}{A}}{A}}
  \end{equation*}
  is an isomorphism in $\D[A]$; see \corcite[(2.3.8)]{lnm}. This can
  be verified locally, as $(\delta_M^A)_\p = \delta_{M_\p}^{A_\p}$ for
  all $\p$ in $\Spec{A}$, so we may assume that $A$ is local.
  
  Now, let $K$ be the Koszul complex on a set of generators for the
  maximal ideal $\m$, and let $E$ be the injective hull of $A/\m$. As
  the complex $\RHom[A]{\RHom[A]{M}{A}}{A}$ has degreewise finitely
  generated homology, it follows from \cite[1.3]{HBFSIn03} that
  $\delta_M^A$ is an isomorphism if $\Ltp[A]{\delta_M^A}{K}$ is one.
  Set $J=\Hom[A]{K}{E}$, and note that this is a bounded complex of
  injective modules and has homology modules of finite length.  By
  \lemref{nonstandardTEV} there is an isomorphism:
  \begin{equation*}
    \dqisdef{\omega_{MAJ}}{\Ltp[A]{\RHom[A]{M}{A}}{J}}{\RHom[A]{M}{J}}.
  \end{equation*}
  Furthermore, as $K$ has homology modules of finite length, the
  biduality morphism
  \begin{equation*}
    \mapdef{\delta_K^E}{K}{\Hom[A]{\Hom[A]{K}{E}}{E}}
  \end{equation*}
  is an isomorphism in $\D[A]$. The target complex is isomorphic to
  $\RHom[A]{J}{E}$, and there is a commutative diagram in
  $\D[A]$\pagebreak
  \begin{equation*}
    \xymatrix@C=6em{
      \Ltp[]{M}{K} 
      \ar[d]_-{\Ltp[]{M}{\delta_K^E}}^-{\qis} 
      \ar[r]^-{\Ltp[]{\delta_M^A}{K}}
      & \Ltp[]{\RHom[]{\RHom[]{M}{A}}{A}}{K}
      \ar[d]_-{\qis}^-{\omega_{\RHom[]{M}{A}AK}}
      \\ 
      \Ltp[]{M}{\RHom[]{J}{E}}
      \ar[dd]_-{\theta_{MJE}}^-{\qis} 
      & \RHom[A]{\RHom[A]{M}{A}}{K}
      \ar[d]^{\RHom[]{\RHom[]{M}{A}}{\delta_K^E}}_-{\qis}
      \\ 
      & \RHom[]{\RHom[]{M}{A}}{\RHom[]{J}{E}}
      \ar[d]_-{\qis}
      \\
      \RHom[]{\RHom[]{M}{J}}{E}
      \ar[r]^-{\qis}_-{\RHom[]{\omega_{MAJ}}{E}} 
      & \RHom[]{\Ltp[]{\RHom[]{M}{A}}{J}}{E}.
    }
  \end{equation*}
  The unlabeled isomorphism is adjointness. The morphism
  $\omega_{\RHom[]{M}{A}AK}$ is an isomorphism by
  \prpcite[2.1(v)]{LWCHHlb}, and the Hom-evaluation morphism
  $\theta_{MJE}$ is an isomorphism by \lemcite[4.4.(I)]{LLAHBF91}. It
  follows that $\Ltp[A]{\delta_M^A}{K}$ is an isomorphism.
  \hfill$\ulcorner\mspace{-5mu}\lrcorner$
\end{pfofa}

To prove the non-commutative part of \thmref{g} we need the following:

\begin{splem}
  \label{lem:g}
  Let $A$ be a two-sided noetherian ring with a dualizing complex; see
  {\rm \pgref{dc}}. An $A$-complex $M$ with bounded and degreewise
  finitely generated homology has finite G-dimension if and only if
  the complex $\RHom[A]{M}{A}$ is homologically bounded and the
  biduality morphism
  $\mapdef{\delta_M^A}{M}{\RHom[\Aop]{\RHom[A]{M}{A}}{A}}$ is an
  isomorphism in~$\D[A]$\footnote{By \pgref{dc} the dualizing complex
    $D$ has resolutions ${}_AP_A \qra {}_AD_A \qra {}_AI_A$ by
    $A$-bimodules, where each module in $P$ is projective over both
    $A$ and $\Aop$, and each module in $I$ is injective over both $A$
    and $\Aop$.  It follows that $A$ has a resolution $A \qra
    J=\Hom[A]{P}{I}$ by $A$-bimodules, where each module in $J$ is
    injective over both $A$ and $\Aop$. Consequently, $\delta_M^A$ is
    represented by $M \to \Hom[\Aop]{\Hom[A]{M}{J}}{J}$.}.
\end{splem}

\begin{proof}
  By \prpcite[3.8(b) and thm.~4.1]{CFH-06} the complex $M$ has finite
  G-dimension if and only if the complex $\Ltp[A]{D}{M}$ is
  homologically bounded and the natural morphism
  $\mapdef{\eta_M}{M}{\RHom[A]{D}{\Ltp[A]{D}{M}}}$ is an isomorphism
  in $\D[A]$. The next two isomorphisms are adjointness and Hom
  evaluation; see \lemcite[4.4.(I)]{LLAHBF91}.
  \begin{align}
    \RHom[A]{M}{A} &\qis \RHom[A]{\Ltp[A]{D}{M}}{D}\quad\text{and}\\
    \label{eq:i2}
    \Ltp[A]{D}{M} &\qis \RHom[\Aop]{\RHom[A]{M}{A}}{D}.
  \end{align}
  It follows that $\RHom[A]{M}{A}$ is homologically bounded if and
  only if $\Ltp{D}{M}$ is so. The diagram below shows that
  $\delta_M^A$ is an isomorphism if and only if $\eta_M$ is one.
  \begin{equation*}
    \xymatrix{
      \RHom[\Aop]{\RHom[A]{M}{A}}{A} 
      \ar[r]_-{\qis}
      & \RHom[\Aop]{\RHom[A]{M}{A}}{\RHom[A]{D}{D}} 
      \\
      M 
      \ar[u]^-{\delta_M^A} 
      \ar[d]_{\eta_M} 
      &
      \\
      \RHom[A]{D}{\Ltp[A]{D}{M}} 
      \ar[r]^-{\qis} 
      & \RHom[A]{D}{\RHom[\Aop]{\RHom[A]{M}{A}}{D}} 
      \ar[uu]^-{\qis}
    }
  \end{equation*}
  The upper horizontal isomorphism is by definition of a dualizing
  complex, and the lower one is induced by \eqref{i2}; the right
  vertical isomorphism is swap. \qed
\end{proof}

\begin{pfofac}
  Assume that $A$ has a dualizing complex $D$; see \pgref{dc} for the
  definition. By \lemref{g} it suffices, as in the commutative case,
  to show that homological boundedness of $\RHom[A]{M}{A}$ implies
  that $\mapdef{\delta_M^A}{\!\!M}{\RHom[\Aop\!]{\RHom[A]{M}{A}}{A}}$
  is an isomorphism in~$\D[A]$. This follows from the commutative
  diagram below.
  \begin{equation*}
    \xymatrix{
      M
      \ar[r]^-{\delta^A_M}
      \ar[d]^-{\qis}
      & \RHom[\Aop]{\RHom[A]{M}{A}}{A}
      \ar[d]_-{\qis}
      \\
      \Ltp[A]{\RHom[\Aop]{D}{D}}{M}
      \ar[d]^-{\qis}
      & \RHom[\Aop]{\RHom[A]{M}{A}}{\RHom[\Aop]{D}{D}}
      \\
      \RHom[\Aop]{\RHom[A]{M}{D}}{D}
      \ar[r]^-{\qis}
      & \RHom[\Aop]{\Ltp[A]{\RHom[A]{M}{A}}{D}}{D} 
      \ar[u]^-{\qis}
    }
  \end{equation*}
  The vertical isomorphisms on the left follow by definition of a
  dualizing complex \pgref{dc} and by \lemcite[4.4.(I)]{LLAHBF91}. The
  horizontal isomorphism is induced by $\omega_{MAD}$, see
  \lemref{nonstandardTEV}. The vertical isomorphisms on the right
  follow by Hom-tensor adjointness and the definition of a dualizing
  complex.
\end{pfofac}

%%% SECTION 5
\section{Examples}
\label{sec:examples}

We consider three elementary constructions that preserve the AC
property.

\begin{spprp}
  Let $A$ and $B$ be left-noetherian and Morita equivalent rings. If
  $A$ satisfies \ac/\uac, then $B$ satisfies \ac/\uac.
\end{spprp}

\begin{proof}
  There exist bimodules ${}_AP_B$ and ${}_BQ_A$, which are finitely
  generated, projective from both sides, and provide an equivalence
  \begin{equation*}
    \xymatrix@C=6em{
      \Cat{A}{mod}\, 
      \ar@<0.7ex>[r]^-{\tp[A]{Q}{-}} 
      & \,\Cat{B}{mod}.
      \ar@<0.7ex>[l]^-{\tp[B]{P}{-}}
    }
  \end{equation*}
  Moreover, for every $B$-module $N$ there is an isomorphism $N \is
  \Hom[A]{P}{\tp[B]{P}{N}}$; see \seccite[9.5]{wei}. For finitely
  generated $B$-modules $M$ and $N$ it follows that
  \begin{align*}
    \Ext[B]{i}{M}{N} %%
    &\is \HH[-i]{\RHom[B]{M}{\Hom[A]{P}{\tp[B]{P}{N}}}} \\
    &\is \HH[-i]{\RHom[A]{\Ltp[B]{P}{M}}{\Ltp[B]{P}{N}}} \\
    &\is \Ext[A]{i}{\tp[B]{P}{M}}{\tp[B]{P}{N}}. \quad \qed
  \end{align*}
\end{proof}

\begin{spexa}
  \label{exa:morita}
  If $A$ satisfies \ac/\uac, then so does every matrix ring over $A$.
\end{spexa}

\begin{spprp}
  Let $A$ and $B$ be left-noetherian rings. The product ring
  \mbox{$A\!\times\!B$} satisfies \ac/\uac\ if and only if both $A$
  and $B$ satisfy \ac/\uac.
\end{spprp}

\begin{proof}
  Note that \mbox{$A\!\times\!B$} is left-noetherian. There are
  equivalences of categories
  \begin{align*}
    \xymatrix@C=10ex{\Mod{A} \times \Mod{B} \ar@<0.6ex>[r]^-{\times} &
      \ar@<0.6ex>[l]^-{\mathsf{s}} \Mod{A\!\times\!B}, }
  \end{align*}
  with the obvious definition of the functor $\times$. The functor
  $\mathsf{s}$ associates to an \mbox{$A\!\times\!B$}-module $M$ the
  pair $\langle \text{\small $(1,0)$}M,\text{\small $(0,1)$}M\rangle$,
  and to an \mbox{$A\!\times\!B$}-linear map \mbox{$\psi \colon M
    \longrightarrow N$} the pair of restrictions $\psi_{(1,0)} \colon
  \text{\small $(1,0)$}M \longrightarrow \text{\small $(1,0)$}N$ and
  $\psi_{(0,1)} \colon \text{\small $(0,1)$}M \longrightarrow
  \text{\small $(0,1)$}N$. Thus, for every pair $M$, $N$ of
  \mbox{$A\!\times\!B$}-modules, $\mathsf{s}$ induces an isomorphism
  \begin{align*}
    \Hom[A \times B]{M}{N} \cong \Hom[A]{\text{\small
        $(1,0)$}M}{\text{\small $(1,0)$}N} \oplus \Hom[B]{\text{\small
        $(0,1)$}M}{\text{\small $(0,1)$}N}.
  \end{align*}
  The functor $\mathsf{s}$ is exact and preserves projectivity,
  indeed, $\text{\small $(1,0)$}M \cong {}_AA_{A\times
    B}\otimes_{A\times B}M$ and similarly $\text{\small $(0,1)$}M
  \cong {}_BB_{A\times B}\otimes_{A\times B}M$. Thus there are
  isomorphisms
  \begin{align*}
    \Ext[A \times B]{i}{M}{N} \cong \Ext[A]{i}{\text{\small
        $(1,0)$}M}{\text{\small $(1,0)$}N} \oplus
    \Ext[B]{i}{\text{\small $(0,1)$}M}{\text{\small $(0,1)$}N},
  \end{align*}
  for all \mbox{$A\!\times\!B$}-modules $M$ and $N$, and all integers
  $i$.  Clearly, an \mbox{$A\!\times\!B$}-module $X$ is finitely
  generated over \mbox{$A\!\times\!B$} exactly when $\text{\small
    $(1,0)$}X$ and $\text{\small $(0,1)$}X$ are finitely generated
  over $A$ and $B$, respectively.  Straightforward arguments finish
  the proof. \qed
\end{proof}

The Chinese Remainder Theorem now yields:

\begin{spexa}
  If $\mathfrak{a}$ and $\mathfrak{b}$ are proper coprime ideals in a
  commutative noetherian ring $R$, then $R/\mathfrak{a}\mathfrak{b}$
  is AC if and only if both $R/\mathfrak{a}$ and $R/\mathfrak{b}$ are
  AC.
\end{spexa}

The results in \cite{DAJLMS04} show, in particular, that the AC
property does not ascend along flat ring homomorphisms. Descent,
however, is straightforward:

\begin{spprp}
  \label{prp:descent}
  Let $A$ be commutative, and let $B$ be a faithfully flat
  left-noetherian $A$-algebra. If $B$ satisfies \ac/\uac, then $A$
  satisfies \ac/\uac.
\end{spprp}

\begin{proof}
  Note that $B$ has a bimodule structure ${}_{A,B}B$.  Let $M$ and $N$
  be finitely generated $A$-modules. Because $B$ is $A$-flat, one has
  the following chain of isomorphisms, where the second is by
  \lemcite[4.4.(F)]{LLAHBF91} and the third is by adjointness.
  \begin{align*}
    \tp[A]{\Ext[A]{i}{M}{N}}{B} %%
    &\is \HH[-i]{\Ltp[A]{\RHom[A]{M}{N}}{B}} \\
    &\is \HH[-i]{\RHom[A]{M}{\Ltp[A]{N}{B}}} \\
    &\is \HH[-i]{\RHom[B]{\Ltp[A]{M}{B}}{\Ltp[A]{N}{B}}} \\
    &\is \Ext[B]{i}{\tp[A]{M}{B}}{\tp[A]{N}{B}}
  \end{align*}
  The desired conclusion now follows by faithful flatness of $B$ over
  $A$. \qed
\end{proof}

\begin{spexa}
  A commutative noetherian ring $R$ is AC if either $R[X]$ or
  $\pows[R]{x}$ is so. Furthermore, if $(R,\m)$ is local and its
  $\m$-adic completion $\widehat{R}$ is AC, then so is $R$.
\end{spexa}

\begin{sprmk}
  For a commutative noetherian Cohen-Macaulay local ring $R$ one gets
  stronger results \cite{LWCHHld}. Indeed, let $\m$ be the maximal
  ideal of $R$, and let $x\in\m$ be an $R$-regular element. If one of
  the rings $R$, $\Rhat$, $R/(x)$, $\pows[R]{X}$, or
  $\poly[R]{X}_{(\m,X)}$ satisfies \ac/\uac, then they all do.
\end{sprmk}

\appendix

%%%%%%%%%%%%%%%%%%%%%%%%%%%%%%%%%%%%%%%%%%%%%%%%%%%%%%%%%%%%%%%%%%%%%%%%%
\section*{Appendix A Conjectures for rings and algebras}
\stepcounter{section}

The Auslander-Reiten and Tachikawa Conjectures originate in
representation theory of algebras, but they have recently received
considerable attention in commutative algebra; see
e.g.~\cite{ABS-05,CHnGJL04,HSV-04,LMS03}. This appendix provides a
quick guide to these and related conjectures, and it explains, in
greater detail, some of the points raised in the Introduction.

\begin{pac}
  According to \cite{DHp90} and \cite[intro.\ to ch.~V]{mas1},
  Auslander conjectured that every Artin algebra satisfies the
  condition \ac, defined in the Introduction. In \cite{DAJLMS04}
  Jorgensen and \c{S}ega showed that the conjecture fails, even for
  commutative local finite dimensional $k$-algebras: one
  counterexample $(R,\m)$ is Gorenstein with $\m^4=0$, another is not
  Gorenstein and has $\m^3=0$ and $\lgtR=8$. A subsequent short
  construction due to Smal{\o} \cite{SOS06} shows that $k\langle
  x,y\rangle/(x^2, y^2, xy+qyx)$, where $q^n\ne0,1$ for all $n$, does
  not satisfy \ac. Further counterexamples are constructed by Mori in
  \seccite[6]{IMr06}.
\end{pac}

\begin{spipg}
  \label{aclist}
  A commutative noetherian regular ring of infinite Krull dimension
  satisfies \ac\ but not \uac. We do not know of any Artin algebra or
  commutative noetherian local ring with that property. Rings known to
  satisfy \uac\ include:
  \begin{itemlist}
  \item Left-noetherian rings of finite global dimension.
  \item Artin algebras of finite representation type; see
    \seccite[2.3]{DHp90}.
  \item Group algebras of finite groups; this follows from
    \thmcite[2.4]{BCR-90}.\footnote{By the isomorphisms
      $\Ext[kG]{i}{M}{N} \is \Hom[k]{M}{\widehat{H}^i(G,N)}$ for
      $i>0$.}
  \item Rings of finite global repetition index. For example quotients
    $\mathcal{O}/\pi$, where $\mathcal{O}$ is a classical order over a
    discrete valuation ring, and $\pi$ is a uniformizing parameter;
    see \seccite[4]{KRGBHZ98}.
  \item Exterior algebras; see \corcite[2.4]{IMr07}.
  \item Commutative noetherian local rings that are Golod or complete
    intersection; see \prpcite[1.4]{DAJLMS04} and
    \thmcite[4.7]{LLAROB00}.
  \item Commutative noetherian Gorenstein local rings $R$ of
    multiplicity $\codim{R}+2$ or with $\codim{R} \le 4$; see
    \thmcite[3.5]{CHnDAJ03} and \thmcite[3.4]{LMS03}.
  \item The trivial extension of a commutative artinian local ring by
    its residue field; see \corcite[3.5]{SNsYYs}.
  \end{itemlist}
  \noindent Further examples of commutative noetherian local rings
  that satisfy \uac\ are given in
  \prpcite[1.1]{DAJLMS04}\footnote{Where part $(2)$ should read:
    $\operatorname{edim}R - \operatorname{depth}R \le 2$.} and in
  \thmcite[3.7]{CHnDAJ03}.
\end{spipg}

\begin{parc}
  The root of this is the Nakayama Conjecture posed in \cite{TNk58}.
  By work of M{\"u}ller \cite{BJM68}, it can be phrased as follows:
  \begin{condition}{0em}
  \item[] Every finite dimensional $k$-algebra $\Alg$ satisfies the
    following condition:

  \item[\nc] If each term in the minimal injective resolution of
    ${}_\Alg\Alg$ is projective, then $\Alg$ is quasi-Frobenius.
  \end{condition}
  In \cite{MAsIRt75} Auslander and Reiten propose the Generalized
  Nakayama Conjecture:
  \begin{condition}{1em}
  \item[] Every Artin algebra $\Alg$ satisfies the following
    condition:
  \item[\gnc] Every indecomposable injective $\Alg$-module occurs as a
    summand in one of the terms in the minimal injective resolution of
    ${}_\Alg\Alg$.
  \end{condition}
  A finite dimensional $k$-algebra is an Artin algebra, and an Artin
  algebra that satisfies \gnc\ also satisfies \nc,
  cf.~\prpcite[IV.3.1]{rta}. It is proved in \cite{MAsIRt75} that the
  Generalized Nakayama Conjecture is equivalent to:
  \begin{condition}{4em}
  \item[] Every Artin algebra $\Alg$ satisfies the following
    condition:
  \item[\gncm] Every finitely generated
    $\Alg$-generator\footnote{Defined as follows: for every finitely
      generated $\Alg$-module $T$ there is an epimorphism $M' \onto T$
      such that $M' \in \mathsf{add}(M)$.} %
    $M$ with $\Ext[\Alg]{\ge 1}{M}{M} = 0$ is projective.
  \end{condition}
  It is \textsl{not} known if a given finite dimensional $k$-algebra
  satisfies \gnc\ if and only if it satisfies \gncm. What \textsl{is}
  known is that \gnc\ holds for all $k$-algebras if and only if \gncm\
  does; see \cite[remark after thm.~3.4.3]{KYm96}.
  
  In \cite{ADS-93} the condition \gncm\ is considered for any
  noetherian ring, and it is noted that a ring $A$ satisfies \gncm\ if
  and only if it satisfies \arc; see the Introduction.  Indeed, an
  $A$-generator $M$ with $\Ext[A]{\ge 1}{M}{M} = 0$ also has
  $\Ext[A]{\ge 1}{M}{A} = 0$, and for every $A$-module $N$ the module
  $N \oplus A$ is an $A$-generator.
\end{parc}

\begin{spipg}
  \label{arclist}
  Rings know to satisfy \arc\ include:
  \begin{itemlist}
  \item Left-noetherian rings over which every finitely generated
    module has an ultimately closed projective
    resolution;\footnote{Defined as a degreewise finitely generated
      projective resolution for which there is a $d >0$ such that the
      $d$th syzygy has a decomposition whose factors are summands of
      earlier syzygies; see \seccite[3]{JPJ61}.} %
     see \prpcite[1.3]{MAsIRt75}.
  \item Rings $\Lambda/(\pmb{x})\Lambda$ where $\Lambda$ is a
    noetherian algebra of finite global dimension over a commutative
    noetherian complete local ring $(R,\m)$, and $\pmb{x} \in \m$ is a
    $\Lambda$-sequence; see \prpcite[1.9]{ADS-93}. In particular,
    commutative noetherian complete intersection local rings; see also
    \thmcite[4.2]{LLAROB00}.
  \item Group algebras $kG$, where $G$ is a finite group and $k$ is a
    field of characteristic $p>0$; see \cite[5.2.3]{ben2}.
  \item Commutative artinian local rings $(R,\m)$ with
    $2\lgt{(\SocR)}>\lgtR$ or with $\m^3=0$; see  \cite[4.3]{HLDOVl} and
    \thmcite[4.1]{HSV-04}.
  \item Commutative noetherian Golod local rings; see
    \prpcite[1.4]{DAJLMS04}.
  \item Rings $R/(\pmb{x})$ where $\pmb{x}$ is an $R$-sequence, and
    $R$ is commutative, noetherian, local, excellent, Cohen-Macaulay,
    normal, and either Gorenstein or a $\QQ$-algebra. This is a
    special case of \thmcite[0.1]{CHnGJL04}.
  \item Commutative noetherian Gorenstein local rings $R$ with
    $\codim{R} \le 4$; see \corcite[3.5]{LMS03}.
  \end{itemlist}
\end{spipg}

\begin{ptc}
  The conditions above relate to two conjectures of Tachikawa
  \cite[\S8]{HTc73}:

  \begin{condition}{4em}
  \item[] Every finite dimensional $k$-algebra $\Alg$ satisfies the
    following condition:
  \item[\tcis] If $\Ext[\Alg]{\ge 1}{\Hom[k]{\Alg_\Alg}{k}}{\Alg} =
    0$, then $\Alg$ is quasi-Frobenius.\footnote{The conjecture on
      p.~115 in \cite{HTc73} is equivalent to this one by the
      arguments on p.~114 ibid.}%

  \item[and \hfill]\vspace{\thmbotspace}

  \item[] Every quasi-Frobenius finite dimensional $k$-algebra $\Alg$
    satisfies:

  \item[\tciis] Every finitely generated $\Alg$-module $M$ with
    $\Ext[\Alg]{\ge 1}{M}{M} = 0$ is projective.
  \end{condition}
  It is proved in \cite{HTc73} and \cite{KYm96} that the Nakayama
  Conjecture holds if and only if both Tachikawa Conjectures hold. The
  diagram below depicts the known relations between conditions on
  finite dimensional $k$-algebras.
  \begin{equation*}
    \SelectTips{cm}{}
    \tag{A.3} \stepcounter{theorem}
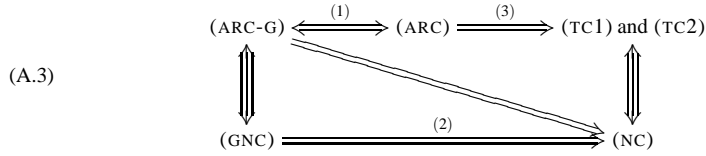
    \begin{gathered}
      \xymatrix@C=3pc@R=2.5pc{ \gncm \ar@{<=>}[r]^-{(1)} \ar@3{<->}[d]
        \ar@{=>}[rrd] & \arc \ar@{=>}[r]^-{(3)} & \tcis \text{ and }
        \tciis \ar@3{<->}[d]
        \\
        \gnc \ar@{=>}[rr]^-{(2)} & & \nc }
    \end{gathered}\\[-.5ex]
  \end{equation*}
  \begin{center}
    \parbox{0.809\textwidth}{\footnotesize The notation
      \propp$\Rightarrow$\propq\ means that every algebra that
      satisfies \propp\ also satisfies \propq, while
      \propp$\Lleftarrow\mspace{-6mu}\Rrightarrow$\propq\ means that
      \textsl{all} algebras satisfy \propp\ if and only \textsl{all}
      algebras satisfy~\propq.}
  \end{center}
  \noindent The implications $(1)$ and $(2)$ were discussed above; the
  implication $(3)$ is clear; cf.\ the proof of \prpref{tachikawa}.
  The remaining implications are proved in
  \thmcite[3.4.3]{KYm96}.\footnote{The remark following
    \thmcite[3.4.3]{KYm96} indicates that \textsl{any given algebra}
    satisfies \nc\ if and only if it satisfies
    (\textsc{tc}{\scriptsize 1}) and (\textsc{tc}{\scriptsize
      2}). However, this strong statement is \textsl{not} known to be
    true, cf.\ thm.~3.4.2 ibid. We thank Professor Yamagata for
    clarifying this to us.}%
  
  In commutative algebra, Avramov, Buchweitz, and \c{S}ega
  \cite{ABS-05} make a conjecture related to the first of Tachikawa's
  conjectures mentioned above. Their conjecture is the following:
  \begin{condition}{4em}
  \item[] Every commutative noetherian Cohen-Macaulay local ring $R$
    satisfies:
  \item[\ltc] If $R$ has a dualizing module $D$ and $\Ext{\ge 1}{D}{R}
    = 0$, then $R$ is Gorenstein.
  \end{condition}
  It is clear that both conditions \gnc\ and \nc\ make sense for, and
  are satisfied by, every commutative noetherian local ring $R$.
  However, the conjecture of Avramov, Buchweitz, and \c{S}ega is still
  open, even in the case where $R$ is a finite dimensional
  $k$-algebra.  This emphasizes the point that the implication
  \mbox{\nc$\equiv\mspace{-4mu}\Rrightarrow$\tci} in (A.3) is not
  known to restrict to commutative local $k$-algebras. A list of rings
  that satisfy \ltc\ is provided in \cite[intro.\ and sec.~9]{ABS-05}.
\end{ptc}

\begin{spipg}
  \label{ourlist}
  We end this appendix by summarizing a couple of contributions of
  this paper.

  \thmref{ar} is new, even for finite dimensional $k$-algebras.  In
  particular, it adds exterior algebras and rings of finite global
  repetition index to the list of rings known to satisfy \arc.
      
  \prpref{tachikawa} shows that \ac\ implies a generalized version of
  \tcis\ for two-sided noetherian rings with a dualizing complex.
\end{spipg}

%%%%%%%%%%%%%%%%%%%%%%%%%%%%%%%%%%%%%%%%%%%%%%%%%%%%%%%%%%%%%%%%%%%%%%%%%

\section*{Appendix B AB rings}
\stepcounter{section}

Huneke and Jorgensen \cite{CHnDAJ03} introduce \emph{AB rings} as
commutative noetherian Gorenstein local rings that satisfy
\uac---equivalently \ac, cf.~\prpref{equivalentAB}.  Our
\lemref{nonstandardTEV} is inspired by ideas in \cite{CHnDAJ03}; in
particular by \prpcite[5.2 and 5.5]{CHnDAJ03}. In this appendix we
apply \lemref{nonstandardTEV} to reestablish two main results
\thmcite[4.1 and cor.~4.2]{CHnDAJ03} in the setting of complexes over
a commutative noetherian ring $R$ with $\id{R}$~finite.

In the following we use the term \emph{totally reflexive} for a module
that is either $0$ or of G-dimension $0$ in the sense of Auslander and
Bridger \cite{MAsMBr69}, cf.~the Introduction.

\begin{splem}
  \label{lem:bounded}
  Let $R$ be commutative and noetherian with $\id{R}$ finite, and let
  $M$ be an $R$-complex. If $M$ is isomorphic in $\D$ to a complex of
  totally reflexive $R$-modules, then the biduality morphism
  $\delta_M^R$ is invertible:
  \begin{equation*}
    M \qra \RHom{\RHom{M}{R}}{R}.
  \end{equation*}
  In particular, $M$ is homologically bounded if and only if\,
  $\RHom{M}{R}$ is so.
\end{splem}

\begin{proof}
  Let $G$ be a complex of totally reflexive $R$-modules such that
  there is an isomorphism $M \qis G$ in $\D$; further let $\alpha
  \colon R \qra I$ be a bounded injective resolution.  We start by
  proving that the complex $\Hom{G}{R}$ is isomorphic to $\RHom{M}{R}$
  in $\D$. We do so by arguing that $\Hom{G}{-}$ preserves the
  quasiisomorphism $\alpha$, that is, we show exactness of the complex
  \begin{displaymath}
    \Cone{\Hom{G}{\alpha}} \is \Hom{G}{\Cone{\alpha}}.
  \end{displaymath}
  Note that $\Cone{\alpha}$ is a bounded and exact complex of modules
  of finite injective dimension. Thus, for every $\d$ the complex
  $\Hom{G_\d}{\Cone{\alpha}}$ is exact by \corcite[(2.4.4)(a)]{lnm},
  and the claim follows by \lemcite[2.4]{CFH-06}. The complex
  $\Hom{G}{R}$ consists of totally reflexive $R$-modules, see
  \cite[obs.~(1.1.7)]{lnm}, so the argument above applies to show that
  $\Hom{\Hom{G}{R}}{R}$ is isomorphic to $\RHom{\RHom{M}{R}}{R}$ in
  $\D$.  Consequently, the morphism
  \begin{equation*}
    \dmapdef{\delta_M^R}{M}{\RHom{\RHom{M}{R}}{R}}
  \end{equation*}
  in $\D$ is represented by
  \begin{equation*}
    \dmapdef{\delta_G^R}{G}{\Hom{\Hom{G}{R}}{R}},
  \end{equation*}
  which is an isomorphism of $R$-complexes, as each module $G_\d$ is
  totally reflexive. \qed
\end{proof}

In the next two results, we use the notation $(-)^\star =
\RHom{-}{R}$.

\begin{theorem}
  \label{thm:CD}
  Let $R$ be commutative and noetherian with $\id{R}$ finite and
  assume that $R$ satisfies \ac. For $R$-complexes $M$ and $N$ with
  bounded and degreewise finitely generated homology the following
  conditions are equivalent:
  \begin{eqc}
  \item $\RHom{M}{N}$ is homologically bounded.
  \item $\RHom{N}{M}$ is homologically bounded.
  \item $\Ltp{M^\star}{N}$ is homologically bounded.
  \end{eqc}
\end{theorem}

\begin{proof}
  We prove the implications $(i)\Ra(iii)\Ra(ii)$, then $(ii) \Ra (i)$
  by symmetry.
  
  Homological boundedness of $\RHom{M}{N}$ yields by
  \prpref{equivalentAB} an isomorphism $\Ltp{M^\star}{N} \qis
  \RHom{M}{N}$ in $\D$. This shows the first implication.
  
  For the second implication, note that there are isomorphisms
  \begin{align*}
    \Ltp{M^\star}{N} \qis \Ltp{N}{\RHom{M}{R}} \qis
    \RHom{\RHom{N}{M}}{R},
  \end{align*}
  where the last one uses finiteness of $\id{R}$, see
  \cite[(1.4)]{HBF79}. Thus, the complex $\RHom{\RHom{N}{M}}{R}$ is
  homologically bounded, and \lemref{bounded} finishes the proof once
  we show that $\RHom{N}{M}$ is isomorphic in $\D$ to a complex of
  totally reflexive modules. To this end, let \mbox{$L \qra N$} be a
  degreewise finitely generated free resolution and choose a bounded
  complex $G$ of totally reflexive modules such that $G \qis M$; see
  \thmcite[(2.3.7)]{lnm} and \cite[(1.4)]{HBF79}.  The complex
  $\Hom{L}{G}$ is isomorphic to $\RHom{N}{M}$ in $\D$ and consists of
  totally reflexive modules. \qed
\end{proof}

\begin{spcor}
  Let $R$ be commutative and noetherian with $\id{R}$ finite and
  assume that $R$ satisfies \ac. For $R$-complexes $M$ and $N$ with
  bounded and degreewise finitely generated homology the following
  conditions are equivalent:
  \begin{eqc}
  \item $\Ltp{M}{N}$ is homologically bounded.
  \item $\RHom{M^\star}{N}$ is homologically bounded.
  \item $\RHom{N^\star}{M}$ is homologically bounded.
  \end{eqc}
\end{spcor}

\begin{proof}
  From the isomorphisms $N \qis N^{\star\star}$ and $M \qis
  M^{\star\star}$, see \thmcite[(2.3.14)]{lnm}, it follows that the
  complexes in \eqclbl{ii} and \eqclbl{iii} are isomorphic by swap. By
  \thmref{CD} condition \eqclbl{ii} holds if and only if the complex
  $\Ltp{M^{\star\star}}{N} \qis \Ltp{M}{N}$ is homologically
  bounded. \qed
\end{proof}

\begin{acknowledgements}
  We thank Petter Andreas Bergh and Jiaqun Wei for helpful comments on
  an earlier version of the paper.
\end{acknowledgements}

%%% BIBLIOGRAPHY
\bibliographystyle{spmpsci}      % mathematics and physical sciences
%\bibliography{../../+references}
  \newcommand{\arxiv}[2][AC]{\mbox{\href{http://arxiv.org/abs/#2}{\sf arXiv:#2
  [math.#1]}}}
  \newcommand{\oldarxiv}[2][AC]{\mbox{\href{http://arxiv.org/abs/math/#2}{\sf
  arXiv:math/#2
  [math.#1]}}}\providecommand{\MR}[1]{\mbox{\href{http://www.ams.org/mathscine%
t-getitem?mr=#1}{#1}}}
  \renewcommand{\MR}[1]{\mbox{\href{http://www.ams.org/mathscinet-getitem?mr=#%
1}{#1}}}

\end{document}